%% file: systoles.tex
\documentclass[11pt]{article}
\usepackage[twoside,a4paper,margin=1in]{geometry}
\usepackage{bm,bbm,amsmath,amssymb,amsthm,amsfonts,hyperref,url}
\usepackage{inkfrag}  
\usepackage[noadjust]{cite}

\graphicspath{{fig/}}

\newcommand{\executeiffilenewer}[3]{%
\ifnum\pdfstrcmp{\pdffilemoddate{#1}}%
{\pdffilemoddate{#2}}>0%
{\immediate\write18{#3}}\fi%
}

\newcommand{\Set}[1]{\left\{ #1 \right\}}

\makeatletter
\newcommand{\Bigbar}[1]{\mathrel{\left|\vphantom{#1}\right.\n@space}}
\makeatother
\newcommand{\Setbar}[2]{\Set{#1 \Bigbar{#1 #2} #2}}

\definecolor{blueblack}{rgb}{0,0,.7}
\newcommand{\emphdef}[1]{%
  \textcolor{blueblack}{%
    \textbf{\emph{#1}}%
  }%
}

\newtheorem{theorem}{Theorem}[section]
\newtheorem{lemma}[theorem]{Lemma}

\newtheorem{corollary}[theorem]{Corollary}

\newtheorem{question}[theorem]{Question}
\newtheorem{proposition}[theorem]{Proposition}

\DeclareMathOperator{\vol}{vol}

\newcommand{\sys}{\ensuremath{\text{sys}}}

\newcommand{\Q}{\mathbb{Q}}
\newcommand{\R}{\mathbb{R}}
\newcommand{\SSS}{\mathbb{S}}

\newcommand{\ZZ}{\mathbb{Z}}
\renewcommand{\epsilon}{\varepsilon}
\renewcommand{\phi}{\varphi}

\newcommand\surf{\ensuremath{S}}

\title{Discrete Systolic Inequalities and Decompositions\\ of Triangulated
  Surfaces%
  \thanks{Supported by the French ANR Blanc project ANR-12-BS02-005 (RDAM).
    Preliminary version in \emph{Proceedings of the 30th Annual Symposium
      on Computational Geometry}, 2014.}%
}

\author{%
  \'Eric Colin de Verdi\`ere\thanks{CNRS, D\'epartement d'informatique,
    \'Ecole normale sup\'erieure, Paris, France. Email:
    \protect\url{eric.colin.de.verdiere@ens.fr}.}%
  \and %
  Alfredo Hubard\thanks{Laboratoire de l'Institut Gaspard Monge,
    Universit\'e Paris-Est Marne-la-Vall\'ee.  Email:
    \protect\url{alfredo.hubard@inria.fr}.  Portions of this work were done
    during a post-doctoral visit at the D\'epartement d'informatique of
    \'Ecole normale sup\'erieure, funded by the Fondation Sciences
    Math\'ematiques de Paris.}%
  \and %
  Arnaud de Mesmay\thanks{IST Austria, Vienna.  Email:
    \protect\url{arnaud.de.mesmay@ist.ac.at}.  Portions of this work were done
    as a Ph.D.\ student at the D\'epartement d'informatique of
    \'Ecole normale sup\'erieure.}%
}

\begin{document}
\begin{titlepage}
\maketitle
\begin{abstract}
  How much cutting is needed to simplify the topology of a surface? We
  provide bounds for several instances of this question, for the minimum
  length of topologically non-trivial closed curves, pants decompositions,
  and cut graphs with a given combinatorial map in triangulated
  combinatorial surfaces (or their dual cross-metric counterpart).

  Our work builds upon Riemannian systolic inequalities, which bound the
  minimum length of non-trivial closed curves in terms of the genus and the
  area of the surface. We first describe a systematic way to translate
  Riemannian systolic inequalities to a discrete setting, and
  vice-versa. This implies a conjecture by Przytycka and Przytycki from
  1993, a number of new systolic inequalities in the discrete setting, and
  the fact that a theorem of Hutchinson on the edge-width of triangulated
  surfaces and Gromov's systolic inequality for surfaces are essentially
  equivalent. We also discuss how these proofs generalize to higher
  dimensions.
  
  Then we focus on topological decompositions of surfaces. Relying on ideas
  of Buser, we prove the existence of pants decompositions of length
  $O(g^{3/2}n^{1/2})$ for any triangulated combinatorial surface of genus
  $g$ with $n$ triangles, and describe an $O(gn)$-time algorithm to compute
  such a decomposition.

  Finally, we consider the problem of embedding a cut graph (or more
  generally a cellular graph) with a given combinatorial map on a given
  surface.  Using random triangulations, we prove (essentially) that, for
  any choice of a combinatorial map, there are some surfaces on which any
  cellular embedding with that combinatorial map has length superlinear in
  the number of triangles of the triangulated combinatorial surface.  There
  is also a similar result for graphs embedded on polyhedral triangulations.
\end{abstract}
\end{titlepage}

\section{Introduction}

Shortest curves and graphs with given properties on surfaces have been much
studied in the recent computational topology literature; a lot of effort
has been devoted towards efficient algorithms for finding shortest curves
that simplify the topology of the surface, or shortest topological
decompositions of
surfaces~\cite{eh-ocsd-04,ew-gohhg-05,k-csntc-06,cm-fsnnc-07,ccl-aeweg-12,ew-csec-10,efn-gmcse-12,en-crpse-11}
(refer also to the recent surveys~\cite{c-tags-12,e-cocb-12}).  These
objects provide ``canonical'' simplifications or decompositions of
surfaces, which turn out to be crucial for algorithm design in the case of
surface-embedded graphs, where making the graph planar is
needed~\cite{cen-hfcc-12,cen-mcshc-09,ls-ggcg-10,cce-msspe-13}, as well as
for many purposes in computer graphics and mesh
processing~\cite{gw-tnr-01,lgq-smcpd-09,lm-ndtms-98,pb-stmss-00,whds-retfi-04}.

In this article, we study inequalities that relate the size of a
triangulated surface with the length of such shortest curves and graphs
embedded thereon. The model parameter that we study is the notion of
\emph{edge-width} of an (unweighted) graph embedded on a
surface~\cite{rs-gmsdp-88,ccl-aeweg-12}, that is, the length of a shortest
closed walk in the graph that is non-contractible on the surface (i.e.,
cannot be deformed to a single point on the surface).  In particular we are
interested in the following question: What is the largest possible
edge-width, over all triangulations with $n$~triangles of an orientable
surface of genus~$g$ without boundary? It was known~\cite{h-snceg-88} that
$O(\sqrt{n/g}\log g)$ is an upper bound for the edge-width, and we prove
that this bound is asymptotically tight, namely, that some combinatorial
surfaces of arbitrarily large genus achieve this bound.  We also study
similar questions for other types of curves (non-separating closed curves,
null-homologous but non-contractible closed curves) and for decompositions
(pants decompositions, and cut graphs with a prescribed combinatorial map),
and give an algorithm to compute short pants decompositions.

Most of our results build upon or extend to a discrete setting some known
theorems in \emph{Riemannian systolic geometry}, the archetype of which is
an upper bound on the systole (the length of a shortest non-contractible
closed curve---a continuous version of the edge-width) in terms of the
square root of the area of a Riemannian surface without boundary (or more
generally the $d$th root of the volume of an essential Riemannian
$d$-manifold).  Riemannian systolic geometry~\cite{g-sii-92,k-sgt-07} was
pioneered by Loewner and Pu~\cite{p-sicnr-52}, reaching its maturity with
the deep work of Gromov~\cite{g-frm-83}. In Thurston's words, topology is
naked and it \emph{dresses} with geometric structures; systolic geometry
regards the lengths and areas of all those possible outfits.  Similarly,
endowing a topological surface with a triangulation is a way to ``dress''
it and much of this paper leverages on comparing these two types of
outfits.

We always assume that the surface has \emph{no boundary}, that the
underlying graph of the combinatorial surface is a \emph{triangulation},
and that its edges are \emph{unweighted}; the curves and graphs we seek
remain on the edges of the triangulation.  Lifting any of these
restrictions invalidates or significantly worsens our bounds.  In many
natural situations, such requirements hold, such as in geometric modeling
and computer graphics, where triangular meshes of surfaces without boundary
are typical and, in many cases, the triangles have bounded aspect ratio
(which immediately implies that our bounds apply, the constant in the
$O(\cdot)$ notation depending on the aspect ratio).

After the preliminaries (Section~\ref{S:prelim}), we prove three
independent results (Sections \ref{S:alfredo}--\ref{S:eric}), which are
described and related to other works below.  This paper is organized so as
to showcase the more conceptual results before the more technical ones.
Indeed, the results of Section~\ref{S:alfredo} exemplify the strength of
the connection with Riemannian geometry, while the results in
Sections~\ref{S:arnaud} and~\ref{S:eric} are perhaps a bit more specific,
but feature deeper algorithmic and combinatorial tools.

\paragraph*{Systolic inequalities for closed curves on triangulations.}
Our first result (Section~\ref{S:alfredo}) gives a systematic way of
translating a systolic inequality in the Riemannian case to the case of
triangulations, and vice-versa.  This general result, combined with known
results from systolic geometry, immediately implies bounds on the length of
shortest curves with given topological properties: On a triangulation of
genus~$g$ with $n$~triangles, some non-contractible (resp., non-separating,
resp., null-homologous but non-contractible) closed curve has length
$O(\sqrt{n/g}\log g)$, and, moreover, this bound is best possible.

These upper bounds are new, except for the non-contractible case, which was
proved by Hutchinson~\cite{h-snceg-88} with a worse constant in the
$O(\cdot)$ notation.  The optimality of these inequalities is also
new. Actually, Hutchinson~\cite{h-snceg-88} had conjectured that the
correct upper bound was $O(\sqrt{n/g})$; Przytycka and Przytycki refuted
her conjecture, building, in a series of
papers~\cite{pp-schrt-97,pp-stsnc-93,pp-lbsnc-90}, examples that show a
lower bound of $\Omega(\sqrt{n \log g/g})$.  They conjectured in
1993~\cite{pp-stsnc-93} that the correct bound was $O(\sqrt{n/g}\log g)$;
here, we confirm this conjecture.

In Appendix~\ref{A:hd}, we observe that the proofs of the results mentioned
above extend to higher dimensions.  However, the situation is not quite as
symmetrical as in the two-dimensional case: It turns out that discrete
systolic inequalities in terms of the number of vertices or facets imply continuous
systolic inequalities, and that continuous systolic inequalities imply
discrete systolic inequalities only in terms of the number of facets. This
allows us to derive that a systolic inequality in terms of the number of
facets holds for every triangulation of an essential manifold.

As pointed out to us by a referee, slight variations of the results of
Section~\ref{S:alfredo} and Appendix~\ref{A:hd} were simultaneously and
independently discovered by Ryan Kowalick in his Ph.D.\
thesis~\cite{k-dsi-13}. Our approach in Section~\ref{S:conttodiscr} is
similar to his.  In contrast, we use Voronoi diagrams in
Section~\ref{S:discrtocont}, while he uses a different construction
inspired by Whitney.  We will make some further technical comments on his
work at the end of Appendix~\ref{A:hd}.

\paragraph*{Short pants decompositions.}
A pants decomposition is a set of disjoint simple closed curves that split
the surface into \emph{pairs of pants}, namely, spheres with three boundary
components.  In Section~\ref{S:arnaud}, we focus on the length of the
shortest pants decomposition of a triangulation.  As in all previous works,
we allow several curves of the pants decomposition to run along a given
edge of the triangulation. (Formally, we work in the cross-metric surface
that is dual to the triangulation.)

The problem of computing a shortest pants decomposition has been considered
by several authors~\cite{pt-pdpp-06,e-stssc-09}, and has found satisfactory
solutions (approximation algorithms) only in very special cases, such as
the punctured Euclidean or hyperbolic plane~\cite{e-stssc-09}.  Strikingly,
no hardness result is known; the strong condition that curves have to be
disjoint, and the lack of corresponding algebraic structure, makes the
study of short pants decompositions hard~\cite[Introduction]{gpy-pdrs-11}.
In light of this difficulty, it seems interesting to look for algorithms
that compute short pants decompositions, even without guarantee compared to
the optimum solution. 

Inspired by a result by Buser~\cite[Th.~5.1.4]{b-gscrs-92} on short pants
decompositions on Riemannian surfaces, we prove that every triangulation of
genus~$g$ with $n$~triangles admits a pants decomposition of length
$O(g^{3/2}n^{1/2})$, and we give an $O(gn)$-time algorithm to compute one.
While it is known that pants decompositions of length~$O(gn)$ can be
computed for arbitrary combinatorial
surfaces~\cite[Prop.~7.1]{cl-opdsh-07}, the assumption that the surface is
unweighted and triangulated allows for a strictly better bound in the case
where $g=o(n)$.  (It is always true that $g=O(n)$.)  We remark that the
greedy approach coupled with Hutchinson's bound only gives a bound on the
length of the pants decomposition of the form $f(g).\sqrt{n}$ where $f$ is
superpolynomial~\cite[Introduction]{bps-sldsg-12}.
  
On the lower bound side, some surfaces have no pants decompositions with
length $O(n^{7/6-\varepsilon})$, as proved recently by Guth et
al.~\cite{gpy-pdrs-11} using the probabilistic method. Guth et al. show
that polyhedral surfaces obtained by gluing triangles at random have this
property.

\paragraph*{Shortest embeddings of combinatorial maps.}
Finally, in Section~\ref{S:eric}, we consider the problem of decomposing a
surface using a short cut graph with a prescribed combinatorial map.  A
natural approach to build a homeomorphism between two surfaces is to cut
both of them along a cut graph, and to put the remaining disks in
correspondence. However, for this approach to work, cut graphs defining the
same combinatorial map are needed.

   In this direction, Lazarus et
al.~\cite{lpvv-ccpso-01} proved that every surface has a \emph{canonical
  system of loops} (a specific combinatorial map of a cut graph with one
vertex) with length $O(gn)$, which is worst-case optimal, and gave an
$O(gn)$-time algorithm to compute one.

However, there is no strong reason to focus on canonical systems of loops.
It is fairly natural to expect that other combinatorial maps will always
have shorter embeddings (in particular, by allowing several vertices on the
cut graph instead of just one).  Still, we prove (essentially) that for
any choice of combinatorial map of a cut graph, there exist triangulations
with $n$~triangles on which all embeddings of that combinatorial map have a
\emph{superlinear} length, actually $\Omega(n^{7/6-\varepsilon})$.  (Since
$n$ may be~$O(g)$, there is no contradiction with the result by Lazarus et
al.~\cite{lpvv-ccpso-01}.)  In particular, some edges of the triangulation
are traversed $\Omega(n^{1/6-\varepsilon})$ times.

Our proof uses the probabilistic method in the same spirit as the
aforementioned article of Guth et al.~\cite{gpy-pdrs-11}: We show that
combinatorial surfaces obtained by gluing triangles randomly satisfy this
property asymptotically almost surely, i.e., that the probability of
satisfying this property by a random surface tends to one as the number of
triangles tend to infinity. We remark that beyond the extremal qualities
that concern us, random surfaces and their geometry have been heavily
studied recently~\cite{gm-grrs-02,mm-lcgrrs-11} in connection to quantum
gravity~\cite{ps-tcrts-06} and Belyi surfaces~\cite{bm-rcrs-04}.

Another view of our result is via the following problem: Given two graphs
$G_1$ and $G_2$ cellularly embedded on a surface $S$, is there a
homeomorphism $\phi:S\to S$ such that $G_1$ does not cross the image
of~$G_2$ too many times? Our result essentially says that, if~$G_1$ is
fixed, for most choices of trivalent graphs~$G_2$ with $n$~vertices, for
any~$\varphi$, there will be $\Omega(n^{7/6-\varepsilon})$ crossings
between $G_1$ and~$\varphi(G_2)$. This is related to recent
preprints~\cite{mstw-utsnc-13,ghr-epgm-13}, where upper bounds are proved
for the number of crossings for the same problem, but with sets of disjoint
curves instead of graphs. During their proof, Matou\v{s}ek et
al.~\cite{mstw-utsnc-13} also encountered the following problem (rephrased
here in the language of this paper): For a given genus~$g$, does there
exist a \emph{universal} combinatorial map cutting the surface of genus~$g$
into a genus zero surface (possibly with several boundaries), and with a
linear-length embedding on every such surface?  We answer this question in
the negative for cut graphs.  In Appendix~\ref{A:genuszero}, we prove a
related result for families of closed curves cutting the surface into a
genus zero surface.

\section{Preliminaries}\label{S:prelim}

\subsection{Topology for Graphs on Surfaces}

We only recall the most important notions of topology that we will use, and
refer to Stillwell~\cite{s-ctcgt-80} or Hatcher~\cite{h-at-02} for details.
We denote by $S_{g,b}$ the (orientable) surface of \emphdef{genus} $g$ with
$b$~\emphdef{boundaries}, which is unique up to homeomorphism.  The
surfaces $S_{0,0}$, $S_{0,1}$, $S_{0,2}$, and $S_{0,3}$ are respectively
called the \emphdef{sphere}, the \emphdef{disk}, the \emphdef{annulus}, and
the \emphdef{pair of pants}.  Surfaces are assumed to be connected,
compact, and orientable unless specified otherwise.  The notation $\partial
S$ denotes the boundary of~$S$.

A \emphdef{path}, respectively a \emphdef{closed curve}, on a surface $S$
is a continuous map $p: [0,1] \rightarrow S$, respectively $\gamma:\SSS^1
\rightarrow S$. Paths and closed curves are \emphdef{simple} if they are
one-to-one.  A \emphdef{curve} denotes a path or a closed curve.  We refer
to Hatcher~\cite{h-at-02} for the usual notions of homotopy (continuous
deformation) and homology.  A closed curve is \emphdef{contractible} if it
is null-homotopic, i.e., it cannot be continuously deformed to a point.
A simple closed curve is contractible if and only if it bounds a disk.

All the graphs that we consider in this paper are multigraphs, i.e., loops
are allowed and vertices can be joined by multiple edges.  An
\emphdef{embedding} of a graph~$G$ on a surface~$S$ is, informally, a
crossing-free drawing of~$G$ on~$S$.  A graph embedding is
\emphdef{cellular} if its faces are homeomorphic to open disks.  Euler's
formula states that $v-e+f=2-2g-b$ for any graph with $v$~vertices,
$e$~edges, and $f$~faces cellularly embedded on a surface~$S$ with
genus~$g$ with $b$~boundaries.  A \emphdef{triangulation} of a surface~$S$
is a cellular graph embedding such that every face is a triangle.  A
graph~$G$ cellularly embedded on a surface~$S$ yields naturally a
\emphdef{combinatorial map}~$M$, which stores the combinatorial information
of the embedding~$G$, namely, the cyclic ordering of the edges around each
vertex; we also say that $G$ is an \emphdef{embedding} of~$M$ on~$S$.  Two
graphs cellularly embedded on~$S$ have the same combinatorial map if and only if there
exists a self-homeomorphism of~$S$ mapping one (pointwise) to the other.

A graph~$G$ embedded on a surface~$S$ is a \emphdef{cut graph} if the
surface obtained by cutting~$S$ along~$G$ is a disk.  A \emphdef{pants
  decomposition} of $S$ is a family of disjoint simple closed
curves~$\Gamma$ such that cutting $S$ along all curves in $\Gamma$ gives a
disjoint union of pairs of pants.  Every surface $S_{g,b}$ except the
sphere, the disk, the annulus, and the torus admits a pants decomposition,
with $3g+b-3$ closed curves and $2g+b-2$ pairs of pants.

\subsection{Combinatorial and Cross-Metric Surfaces}

We now briefly recall the notions of combinatorial and cross-metric
surfaces, which define a discrete metric on a surface; see Colin de
Verdi\`ere and Erickson~\cite{ce-tnpcs-10} for more details.  In this
paper, all edges of the combinatorial and cross-metric surfaces are
unweighted.

A \emphdef{combinatorial surface} is a surface $S$ together with an
embedded graph $G$, which will always be a triangulation in this
article. In this model, the only allowed curves are walks in~$G$, and the
length of a curve~$c$, denoted by $|c|_G$, is the number of edges of $G$
traversed by~$c$, counted with multiplicity.

However, it is often convenient (Sections~\ref{S:arnaud} and~\ref{S:eric})
to allow several curves to traverse a same edge of~$G$, while viewing them
as being disjoint (implicitly, by ``spreading them apart'' infinitesimally
on the surface).  This is formalized using the dual concept of
\emphdef{cross-metric surface}: Instead of curves in~$G$, we consider
curves in \emphdef{regular} position with respect to the dual graph~$G^*$,
namely, that intersect the edges of~$G^*$ transversely and away from the
vertices; the length of a curve~$c$, denoted by~$|c|_{G^*}$, is the number
of edges of~$G^*$ that $c$~crosses, counted with multiplicity.  Since, in
this article, $G$~is always a triangulation, $G^*$ is always
\emph{trivalent}, i.e., all its vertices have degree three.  Thus, a
cross-metric surface is a surface~$S$ equipped with a cellular, trivalent
graph (usually denoted by~$G^*$).

We note that the previous definition of cross-metric surface is valid also
in the case where the surface has non-empty boundary (see Colin de
Verdi\`ere and Erickson~\cite[Section~1.2]{ce-tnpcs-10} for more details).
Curves and graph embedded on cross-metric surfaces can be manipulated
efficiently~\cite{ce-tnpcs-10}.  The different notions of systoles are
easily translated for both combinatorial and cross-metric surfaces.

Once again, we emphasize that, in this paper, unless otherwise noted,
\emph{\bf all combinatorial surfaces are triangulated (each face is a disk
  with three sides) and unweighted (each edge has weight one)}.  Dually,
\emph{\bf all cross-metric surfaces are trivalent (each vertex has degree
  three) and unweighted (each edge has crossing weight one)}.

\subsection{Riemannian Surfaces and Systolic Geometry}
We will use some notions of Riemannian geometry, referring the interested
reader to standard textbooks~\cite{c-rg-92,k-rg-95}.  A \emphdef{Riemannian
  surface}~$(S,m)$ is a surface~$S$ equipped with a metric~$m$, defined by
a scalar product on the tangent space of every point.  For example, smooth
surfaces embedded in some Euclidean space~$\R^d$ are naturally Riemannian
surfaces---conversely, every Riemannian surface can be isometrically
embedded in some~$\R^d$~\cite{hh-ierme-06} but we will not need this fact.
The length of a (rectifiable) curve~$c$ is denoted by~$|c|_m$.  The
\emphdef{Gaussian curvature}~\emphdef{$\bm{\kappa_p}$} of~$S$ at a
point~$p$ is the product of the eigenvalues of the scalar product at~$p$.
By the Bertrand--Diquet--Puiseux theorem~\cite[Chapter~3,
Prop.~11]{spi-dg-99}, the area of the ball~\emphdef{$\bm{B(p,r)}$} of
radius~$r$ centered at~$p$ equals $\pi r^2-\kappa_p\pi r^4+o(r^4)$. %
We now collect the results from systolic geometry that we will use; for a
general presentation of the field, see, e.g., Gromov~\cite{g-sii-92} or
Katz~\cite{k-sgt-07}.
  \begin{theorem}[{\cite{g-frm-83,ks-esesa-04,g-sii-92,s-absss-08,bs-pmrsl-94}}]\label{sys}
    There are constants $c,c',c'',c'''>0$ such that, on any Riemannian
    surface without boundary, with genus~$g$ and area~$A$:
  \begin{enumerate}
  \item some non-contractible closed curve has length at most $c\sqrt{A/g}
    \log g$;
  \item some non-separating closed curve has length at most
    $c'\sqrt{A/g}\log g$;
  \item some null-homologous non-contractible closed curve has length at
    most $c''\sqrt{A/g}\log g$.
  \end{enumerate}
  Furthermore,
  \begin{enumerate}
  \item[4.] for an infinite number of values of $g$, there exist Riemannian
    surfaces of constant curvature~$-1$ (hence area $A=4\pi(g-1)$) and
    systole larger than $\frac{2}{3\sqrt{\pi}}\sqrt{A/g}\log g-c'''$.  In
    particular, the three previous inequalities are tight up to constant
    factors.
  \end{enumerate}
\end{theorem}
In this theorem, (1) and (2) are due to Gromov~\cite{g-frm-83,g-sii-92}, (3) is
due to Sabourau~\cite{s-absss-08}, and (4) is due to Buser and
Sarnak~\cite[p.~45]{bs-pmrsl-94}.
Furthermore, Gromov's proof yields $c=2/\sqrt{3}$ in (1), which has been
improved asymptotically by Katz and Sabourau~\cite{ks-esesa-04}: They show
that for every $c>1/\sqrt{\pi}$ there exists some integer~$g_c$ so that (1) is valid for every $g\geq g_c$.  

\section{A Two-Way Street}\label{S:alfredo}

In this section, we prove that any systolic inequality regarding closed
curves in the continuous (Riemannian) setting can be converted to the
discrete (triangulated) setting, and vice-versa.

\subsection{From Continuous to Discrete Systolic
  Inequalities}\label{S:conttodiscr}
\begin{theorem}\label{T:conttodiscr}
  Let $(S,G)$ be a triangulated combinatorial surface of genus~$g$, without
  boundary, with $n$~triangles.  Let $\delta>0$ be arbitrarily small.
  There exists a Riemannian metric~$m$ on~$S$ with area~$n$ such that for
  every closed curve $\gamma$ in $(S,m)$ there exists a homotopic closed
  curve $\gamma'$ on~$(S,G)$ with $|\gamma'|_G\leq
  (1+\delta)\sqrt[4]{3}\ |\gamma|_m$.
\end{theorem}

This theorem, combined with known theorems from systolic geometry,
immediately implies:
\begin{corollary}\label{C:newsystolic}
  Let $(S,G)$ be a triangulated combinatorial surface with genus~$g$ and
  $n$~triangles, without boundary.  Then, for some absolute constants~$c$,
  $c'$, and~$c''$:
  \begin{enumerate}
  \item some non-contractible closed curve has length at most
    $c\sqrt{n/g}\log g$;
  \item some non-separating closed curve has length at most
    $c'\sqrt{n/g}\log g$;
  \item some homologically trivial non-contractible closed curve has length
    at most $c''\sqrt{n/g}\log g$.
  \end{enumerate}
\end{corollary}

\begin{proof}[Proof of Corollary~\ref{C:newsystolic}]
  The proof consists in applying Theorem~\ref{T:conttodiscr} to $(S,G)$,
  obtaining a Riemannian metric~$m$.  For each of the different cases, the
  appropriate Riemannian systolic inequality is known, which means that a
  short curve~$\gamma$ of the given type exists on~$(S,m)$
  (Theorem~\ref{sys}(1--3)); by Theorem~\ref{T:conttodiscr}, there exists a
  homotopic curve~$\gamma'$ in~$(S,G)$ such that
  $|\gamma'|_G\le(1+\delta)\sqrt[4]3\ |\gamma|_m$, for any~$\delta>0$.
\end{proof}

Plugging in the best known constants for Theorem~\ref{sys}~(1) allows us to take
$c=2 / \sqrt[4]{3}$, or any $c> \sqrt[4]{3/\pi^2}$ asymptotically using the
refinement of Katz and Sabourau.

Furthermore, we note that, by Euler's formula and double-counting, we have
$n=2v+4g-4$, where $v$ is the number of vertices of~$G$.  Thus, on a
triangulated combinatorial surface with $v\ge g$ vertices, the length of a
shortest non-contractible closed curve is at most
$2\sqrt{2}\sqrt[4]{3}\cdot \sqrt{v/g}\log g < 3.73 \sqrt{v/g}\log g$. This
reproves a theorem of Hutchinson~\cite{h-snceg-88}, except that her proof
technique leads to the weaker constant~$25.27$. This constant can be
improved asymptotically to $\sqrt[4]{108/\pi^2}<1.82$ with the
aforementioned refinement.

\begin{figure}[htb]
\centering
\def\svgwidth{\linewidth}
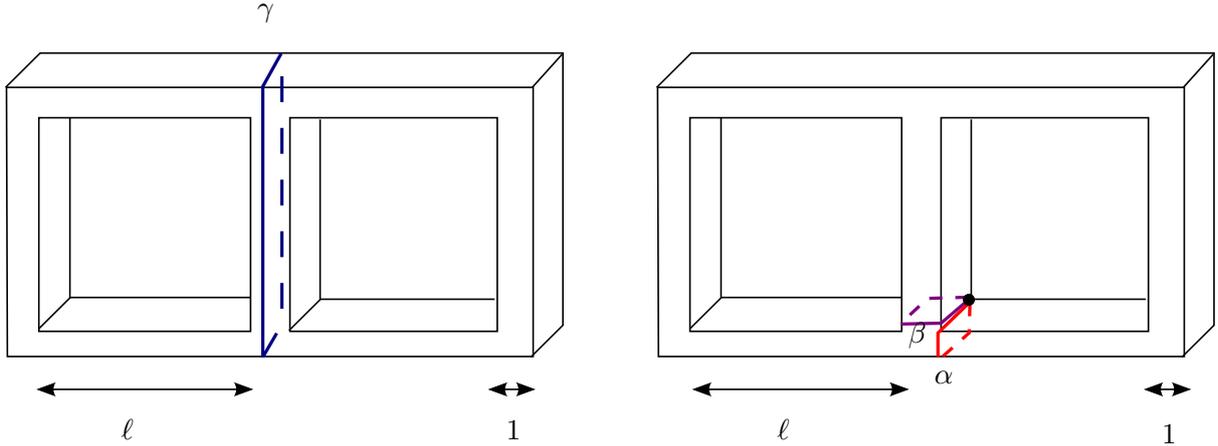%

\caption{A piecewise linear double torus with area~$A$ such that the length
  of a shortest splitting closed curve is $\Omega(A)$ (left), but the
  length of a shortest homologically trivial non-contractible curve,
  concatenation of $\alpha \beta \alpha^{-1} \beta^{-1}$, has length
  $\Theta(1)$.}
\label{fig:counterex}
\end{figure}

We also remark that, in (3), we cannot obtain a similar bound if we require
the curve to be simple (and therefore to be
\emph{splitting}~\cite{ccelw-scsh-08}). Indeed, Figure~\ref{fig:counterex}
shows that the minimum length of a shortest homologically trivial,
non-contractible closed curve can become much larger if we additionally
request the curve to be simple.

\begin{proof}[Proof of Theorem~\ref{T:conttodiscr}]
  We first recall that every surface has a unique structure as a smooth
  manifold, up to diffeomorphism, and we can therefore assume in the
  following that $S$ is a smooth surface.

  The first part of the proof is similar to Guth et
  al.~\cite[Lemma~5]{gpy-pdrs-11}.  Define $m_G$ to be the singular
  Riemannian metric given by endowing each triangle of~$G$ with the
  geometry of a Euclidean equilateral triangle of area~1 (and thus side
  length $2/\sqrt[4]{3}$): This is a genuine Riemannian metric except at a
  finite number of points, the set of vertices of~$G$.  The graph~$G$ is
  embedded on~$(S,m_G)$. Let $\gamma$ be a closed curve $\gamma\colon
  \SSS^1 \to S$. Up to making it longer by a factor at most
  $\sqrt{1+\delta}$, we may assume that $\gamma$ is piecewise linear and
  transversal to~$G$.  Now, for each triangle $T$ and for every maximal
  part~$p$ of~$\gamma$ that corresponds to a connected component
  of~$\gamma^{-1}(T)$, we do the following.  Let $x_0$ and~$x_1$ be the
  endpoints of~$p$ on the boundary of~$T$.  (If~$\gamma$ does not cross any
  of the edges of~$G$, then it is contractible and the statement of the
  theorem is trivial.)  There are two paths on the boundary of~$T$ with
  endpoints $x_0$ and~$x_1$; we replace~$p$ with the shorter of these two
  paths.  Since $T$~is Euclidean and equilateral, elementary geometry shows
  that these replacements at most doubled the lengths of the curve.  Now,
  the new curve lies on the graph~$G$.  We transform it with a homotopy
  into a no longer curve that is an actual closed walk in~$G$, by
  simplifying it each time it backtracks.  Finally, from a closed
  curve~$\gamma$, we obtained a homotopic curve~$\gamma'$ that is a walk
  in~$G$, satisfying $|\gamma'|_G=\sqrt[4]{3}/2\ |\gamma'|_{m_G}\le
  \sqrt{1+\delta}\sqrt[4]3\ |\gamma|_{m_G}$.

  The metric~$m_G$ satisfies our conclusion, except that it has isolated
  singularities.  For the sake of concision we defer the smoothing
  procedure to Lemma~\ref{L:smooth}. This lemma allows us to smooth and
  scale~$m_G$ to obtain a metric~$m$, also with area~$n$, that multiplies
  the length of all curves by at least~$1/\sqrt{1+\delta}$ compared
  to~$m_G$.  This metric satisfies the desired properties.
  \end{proof}

There remains to explain how to smooth the metric, which is done using
partitions of unity.

\begin{lemma}\label{L:smooth}
  With the notations of the proof of Theorem~\ref{T:conttodiscr}, there
  exists a smooth Riemannian metric~$m$ on~$S$, also with area~$n$, such
  that any closed curve~$\gamma$ in~$S$ satisfies
  $|\gamma|_{m}\ge|\gamma|_{m_G}/\sqrt{1+\delta}$.
\end{lemma}
\begin{proof}
  The idea is to smooth out each vertex~$v$ of~$G$ to make $m_G$
  Riemannian, as follows.  Recall that $\delta>0$ is fixed; $\varepsilon>0$
  will be determined later.

  On the open ball $B(v,2\varepsilon)$, consider a Riemannian metric~$m_v$
  such that (i) $m_v$ has area at most~$\delta/3$, and (ii) any path in
  that ball is longer under~$m_v$ than under~$m_G$.  This is certainly
  possible provided $\varepsilon$ is small enough: For example, take any
  diffeomorphism from $B(v,1/2)$ onto the open unit disk~$D$ in the
  plane; define a metric on~$B(v,1/2)$ by taking the pullback metric of
  a multiple~$\lambda$ of the Euclidean metric on~$D$, where $\lambda$ is
  chosen large enough so that this pullback metric is larger than~$m_G$
  (and thus (ii) is satisfied).  If we take $\varepsilon>0$ small enough,
  the restriction of this pullback metric to~$B(v,2\varepsilon)$ also
  satisfies~(i).

  We now use a partition of unity to define a smooth metric $\hat m$ that
  interpolates between~$m_G$ and the metrics~$m_v$. By choosing an
  appropriate open cover, and therefore an appropriate partition of unity
  $\rho$, we obtain a metric $\hat m= \rho_G m_G + \sum_{v \in V} \rho_v
  m_v$ such that:
  \begin{itemize}
  \item outside the balls centered at a vertex $v$  of radius
    $2\varepsilon$, we have $\hat m=m_G$;
  \item inside a ball $B(v,\varepsilon)$, we have $\hat m=m_v$;
  \item in $B(v,2\varepsilon)\setminus B(v,\varepsilon)$, the metric $\hat
    m$ is a convex combination of $m_G$ and~$m_v$.
  \end{itemize}
  The area of~$\hat m$ is at most the sum of the areas of~$m_G$ and
  the~$m_v$'s, which is at most $n(1+\delta)$.  Moreover, for any
  curve~$\gamma$, we have $|\gamma|_{\hat m}\ge|\gamma|_{m_G}$.

  Finally, we scale $\hat m$ to obtain the desired metric~$m$ with
  area~$n$; for any curve~$\gamma$, we indeed have
  $|\gamma|_m\ge|\gamma|_{\hat m}/\sqrt{1+\delta}$.
\end{proof}

\subsection{From Discrete to Continuous Systolic
  Inequalities}\label{S:discrtocont}

Here we prove that, conversely, discrete systolic inequalities imply their
Riemannian analogs.  The idea is to approximate a Riemannian surface by the
Delaunay triangulation of a dense set of points, and to use some recent
results on intrinsic Voronoi diagrams on surfaces~\cite{dzm-ssivd-08}.
\begin{theorem}\label{T:discrtocont}
  Let $(S,m)$ be a Riemannian surface of genus~$g$ without boundary, of
  area~$A$. Let $\delta>0$.  For infinitely many values of~$n$, there
  exists a triangulated combinatorial surface~$(S,G)$ embedded on~$S$ with
  $n$~triangles, such that every closed curve~$\gamma$ in~$(S,G)$ satisfies
  $|\gamma|_m \leq (1+\delta){\sqrt{\frac {32}\pi}}\sqrt{A/n}\ |\gamma|_G$.
\end{theorem}
We have stated this result in terms of the number~$n$ of triangles; in
fact, in the proof we will derive it from a version in terms of the number
of vertices; Euler's formula and double counting imply that, for surfaces,
the two versions are equivalent.  Together with Hutchinson's
theorem~\cite{h-snceg-88}, this result immediately yields a new proof of
Gromov's classical systolic inequality:
\begin{corollary}
  For every Riemannian surface~$(S,m)$ of genus~$g$, without boundary, and
  area~$A$, there exists a non-contractible curve with length at most
  $\frac{101.1}{\sqrt{\pi}} \sqrt{A/g}\log g$.
\end{corollary}
\begin{proof}
  Let $\delta>0$, and let $(S,G)$ be the triangulated combinatorial surface
  implied by Theorem~\ref{T:discrtocont} with $n\ge 6g-4$ triangles.
  Euler's formula implies that the number~$v$ of vertices of~$G$ is at
  least~$g$, hence we can apply Hutchinson's result~\cite{h-snceg-88},
  which yields a non-contractible curve $\gamma$ on~$G$ with
  $|\gamma|_G\leq 25.27 \sqrt{(\frac{n}{2}+2-2g)/g}\log g$.  By
  Theorem~\ref{T:discrtocont}, $|\gamma|_m
  \leq\frac{101.08(1+\delta)}{\sqrt{\pi}}\sqrt{A/g}\log g$.
\end{proof}
On the other hand, using this theorem in the contrapositive together with
the Buser--Sarnak examples (Theorem~\ref{sys}(4)) confirms the conjecture by
Przytycka and Przytycki~\cite[Introduction]{pp-stsnc-93}:
\begin{corollary}\label{C:conj-pp}
  For any $\varepsilon>0$, there exist arbitrarily large $g$ and~$v$ such
  that the following holds: There exists a triangulated combinatorial
  surface of genus~$g$, without boundary, with $v$~vertices, on which the
  length of every non-contractible closed curve is at least
  $\frac{1-\varepsilon}{6}\sqrt{v/g}\log g$. 
\end{corollary}
\begin{proof}  
  Let $\varepsilon>0$, let $(S,m)$ be a Buser--Sarnak surface from
  Theorem~\ref{sys}(4), and let~$G$ be the graph obtained from
  Theorem~\ref{T:discrtocont} from~$(S,m)$, for some $\delta>0$ to be
  determined later.  Combining these two theorems, we obtain that every
  non-contractible closed curve $\gamma$ in~$G$ satisfies
  \[(1+\delta)\sqrt{\frac{32}{\pi}}\sqrt{\frac An}\ |\gamma|_G\ge
  \frac2{3\sqrt\pi}\sqrt{\frac Ag}\log g-c''',\] where $A=4\pi(g-1)$.  If
  $\delta$ was chosen small enough (say, such that $1/(1+\delta)\ge
  1-\varepsilon/2$), and $g$ was chosen large enough, we have
  $|\gamma|_G\ge \frac{1-\varepsilon}{3\sqrt{8}}\sqrt{\frac ng}\log
  g$. Finally, we have $n\ge 2v$ by Euler's formula.
\end{proof}

Before delving into the proof of Theorem~\ref{T:discrtocont}, we introduce
a refinement of the well-known injectivity radius.  The \emph{strong
  convexity radius} at a point~$x$ in a Riemannian surface~$(S,m)$ is the
largest radius $\rho_x$ such that for every $r<\rho_x$ the ball of radius
$r$ centered at $x$ is strongly convex, that is, for any $p,q \in B(x,r)$
there is a unique shortest path in~$(S,m)$ connecting $p$ and $q$, this
shortest path lies entirely within $B(x,r)$, and moreover no other geodesic
connecting $p$ and $q$ lies within $B(x,r)$; we refer to
Klingenberg~\cite[Def.~1.9.9]{k-rg-95} for more details. The strong
convexity radius is positive at every point, and its value on the surface
is continuous (see also Dyer, Zhang, and
M\"oller~\cite[Sect.~3.2.1]{dzm-ssivd-08}).  It follows that for every
compact Riemannian surface~$(S,m)$, the \emph{strong convexity radius}
of~$(S,m)$, namely, the infimum of the strong convexity radii of the points
in~$(S,m)$, is strictly positive.  We will need the following lemma, which
is a result of of Dyer, Zhang, and M\"oller~\cite[Corollary 2]{dzm-ssivd-08}
(see also Leibon~\cite[Theorem~1]{l-rdttat-99} for a very related theorem):
\begin{lemma}\label{L:Delaunay}
  Let $(S,m)$ be a Riemannian surface without boundary, let $\rho'>0$ be
  less than half the strong convexity radius of $(S,m)$, and let $P$ a
  point set of $S$ in general position such that for every $x$ on $S$,
  there exists a point $p$ of $P$ such that $d_m(x,p) \leq \rho'$. Then the
  Delaunay graph of~$P$ is a triangulation of $S$, and its edges are
  shortest paths.
\end{lemma}

\begin{proof}[Proof of Theorem~\ref{T:discrtocont}]
  Let $\eta$, $0<\eta<1/2$ be fixed, and $\varepsilon>0$ to be defined
  later (depending on~$\eta$).  Let $P$ be an $\epsilon$-separated net
  on~$(S,m)$, that is, $P$ is a point set such that any two points in~$P$
  are at distance at least~$\varepsilon$, and every point in~$(S,m)$ is at
  distance smaller than~$\varepsilon$ from a point in~$P$.  For example, if
  we let $P$ be the centers of an inclusionwise maximal family of disjoint
  open balls of radius $\varepsilon/2$, then $P$ is an $\epsilon$-separated
  net. In the following we put $P$ in general position by moving the points
  in~$P$ by at most~$\eta\varepsilon$; in particular, no point in the
  surface is equidistant with more than three points in $P$.

  Let $P=\Set{p_1,\ldots,p_v}$, and let \[V_i:=\Setbar{x\in(S,m)}{\forall
    j\neq i, d(x,p_i)\leq d(x,p_j)}\] be the Voronoi region of~$p_i$.  Since
  every point of~$(S,m)$ is at distance at most~$(1+\eta)\varepsilon$ from
  a point in~$P$, each Voronoi region~$V_i$ is included in a ball of
  radius~$(1+\eta)\varepsilon$ centered at~$p_i$.  Define the Delaunay
  graph of~$P$ to be the intersection graph of the Voronoi regions, and
  note that if $V_i \cap V_j\neq \emptyset$, then the corresponding
  neighboring points of the Delaunay graph are at distance at
  most~$2(1+\eta)\varepsilon$.

  Assume that $\varepsilon$ is small enough so that $(1+\eta)\varepsilon$
  is less than half the strong convexity radius.  Lemma~\ref{L:Delaunay}
  implies that the Delaunay graph, which we denote by~$G$, can be embedded
  as a triangulation of~$S$ with shortest paths representing the edges.

  Consider a closed curve~$\gamma$ on~$G$.  Since neighboring points in~$G$
  are at distance no greater than $2(1+\eta)\varepsilon$ on~$(S,m)$, we
  have $|\gamma|_m \leq 2(1+\eta) \varepsilon |\gamma|_G$.  To obtain the
  claimed bound, there remains to estimate the number~$v$ of points in~$P$.
  By compactness, the Gaussian curvature of $(S,m)$ is bounded from above
  by a constant~$K$.  By the Bertrand--Diquet--Puiseux theorem, the area of
  each ball of radius $\frac{1-2\eta}{2}\varepsilon$ is $\pi(1-2\eta)^2
  \frac{\varepsilon^2}{4}-K\pi
  (1-2\eta)^4\frac{\varepsilon^4}{16}+o(\varepsilon^4)\geq \pi(1-2\eta)^3
  \frac{\varepsilon^2}{4}$ if $\varepsilon>0$ is small enough.  Since the
  balls of radius $(1-2\eta)\frac \varepsilon2$ centered at~$P$ are
  disjoint, their number~$v$ is at most
  $A/(\pi(1-2\eta)^3\frac{\varepsilon^2}{4})$. In other words, $\epsilon
  \leq \frac{2}{\sqrt{\pi(1-2\eta)^3}} \sqrt{A/v}$. Putting together our
  estimates, we obtain that \[|\gamma|_m\leq
  \frac{4(1+\eta)}{\sqrt{\pi(1-2\eta)^3}} \sqrt{\frac A{n/2-2g+2}}\
  |\gamma|_G,\] where $n$~is the number of triangles of~$G$.  Thus, if
  $\varepsilon>0$ is small enough, $n$ can be made arbitrarily large, and
  the previous estimate implies, if $\eta$ was chosen small enough (where
  the dependency is only on~$\delta$) that $|\gamma|_m\le (1+\delta)
  \sqrt{\frac {32}{\pi}}\sqrt{\frac An}\ |\gamma|_G$.
\end{proof}

\paragraph{Remark on orientability.}\ 
Notice that Theorems~\ref{T:conttodiscr} and \ref{T:discrtocont} hold for
non orientable surfaces with the same proofs.  We stated the continuous
systolic inequality for orientable surfaces. As observed by
Gromov~\cite[p.~306]{g-sii-92} a double cover argument shows that the same
results hold (up to a multiplicative constant factor) for the systole of
non-orientable surfaces other than the projective plane. For the projective
plane, a systolic inequality also holds, for which the exact constant is
known and corresponds to metrics of constant positive
curvature~\cite{p-sicnr-52}. Therefore, since our results do not rely on
orientability, the discrete systolic inequalities hold for all surfaces,
with similar dependence on the Euler genus, up to a multiplicative
factor. Notice that when we talked about homology no coefficients were
specified. It is customary to assume $\ZZ$ coefficients for orientable
manifolds and $\mathbb{F}_2$ for non orientable ones.

\section{Computing Short Pants Decompositions}\label{S:arnaud}

Recall that the problem of computing a shortest pants decomposition
for a given surface is open, even in very special cases.  In this
section, we describe an efficient algorithm that computes a short
pants decomposition on a triangulation.  Technically, we allow several
curves to run along a given edge of the triangulation, which is best
formalized in the dual cross-metric setting.  If $g$ is fixed, the
length of the pants decomposition that we compute is of the order of
the square root of the number of vertices:

\begin{theorem}\label{T:mainpart3}
  Let $(S,G^*)$ be a (trivalent, unweighted) cross-metric surface of genus
  $g \geq 2$, with $n$ vertices, without boundary. In $O(gn)$ time, we can
  compute a pants decomposition $(\gamma_1, \ldots, \gamma_{3g-3})$ of~$S$
  such that, for each~$i$, the length of~$\gamma_i$ is at most~$C\sqrt{gn}$
  (where $C$~is some universal constant).
\end{theorem}

 With a little more effort, we can obtain that the length of $\gamma_i$ is at most $C \sqrt{in}$ but we focus on the weaker bound for the sake of clarity. 

The inspiration for this theorem is a result by Buser~\cite{b-gscrs-92},
stating that in the Riemannian case, there exists a pants decomposition
with curves of length bounded by $3 \sqrt{gA}$. The proof of
Theorem~\ref{T:mainpart3} consists mostly of translating Buser's
construction to the discrete setting and making it algorithmic.  The key
difference is that for the sake of efficiency, unlike Buser, we cannot
afford to shorten the curves in their homotopy classes, and we have
to use contractibility tests in a careful manner.  

Given simple, disjoint closed curves~$\Gamma$ in general position on a
(possibly disconnected) cross-metric surface $(S,G^*)$, cutting $S$ along
$\Gamma$, and/or restricting to some connected components, gives another
surface $S'$, and restricting $G^*$ to $S'$ naturally yields a cross-metric
surface that we denote by $(S',G^*_{|S'})$.  To simplify notation we denote
by~$|c|$ (instead of~$|c|_{G^*}$) the length of a curve~$c$ on a
cross-metric surface~$(S,G^*)$.

A key step towards the proof of Theorem~\ref{T:mainpart3} is the following
proposition, which allows us to effectively cut a surface with boundary along
closed curves of controlled length.

\begin{proposition}\label{P:algopantsdecomp-main}
  Let $(S,G^*)$ be a possibly disconnected cross-metric surface, such
  that every connected component has non-empty boundary and admits a
  pants decomposition.  Let $n$ be the number of vertices of~$G^*$ in
  the interior of~$S$.  Assume moreover that $|\partial S|\le\ell$,
  where $\ell$ is an arbitrary positive integer.

  We can compute a family $\Delta$ of disjoint simple closed curves
  of~$(S,G^*)$ that splits~$S$ into one pair of pants, zero, one, or more
  annuli, and another possibly disconnected surface~$S'$ containing no
  disk component, such that $|\partial S'|\leq\ell+4n/\ell+2$. The algorithm takes as
  input~$(S,G^*)$, outputs $\Delta$ and $(S',G^*_{|S'})$, and takes linear
  time in the complexity of~$(S,G^*)$.
\end{proposition}
We first show how Theorem~\ref{T:mainpart3} can be deduced from this
proposition.  It relies on computing a good approximation of the shortest
non-contractible closed curve, cutting along it, and applying
Proposition~\ref{P:algopantsdecomp-main} inductively:

\begin{proof}[Proof of Theorem~\ref{T:mainpart3}]
  To prove Theorem~\ref{T:mainpart3}, we consider our cross-metric surface
  without boundary $(S,G^*)$, and we start by computing a simple
  non-contractible curve~$\gamma$ whose length is at most twice the length
  of the shortest non-contractible closed curve.  Such a curve can be
  computed in $O(gn)$ time~\cite[Prop.~9]{ccl-aeweg-12} (see also Erickson
  and Har-Peled~\cite[Corollary~5.8]{eh-ocsd-04}) and has length at
  most~$C\sqrt{n}$, where $C$ is a universal constant, see
  Section~\ref{S:alfredo}. This gives a surface $S^{(1)}$ with two boundary
  components.
  
Let us define the sequence $\ell_k=C \sqrt{kn}$ for some constant
$C$. We then proceed inductively, applying
Proposition~\ref{P:algopantsdecomp-main} with $\ell=\ell_k$ to
$S^{(k)}$, in order to obtain another surface $S^{(k)\prime}$, from
which we remove all the pairs of pants and annuli. We denote the
resulting surface by $S^{(k+1)}$ and repeat until we obtain a surface
$S^{(m)}$ that is empty. Note that, for every $k$, $S^{(k)}$ contains
no disk, annulus, or pair of pants, and that every application of
Proposition~\ref{P:algopantsdecomp-main} gives another pair of pants.
Therefore, we obtain a pants decomposition of~$S$ by taking the
initial curve~$\gamma$ together with the union of the collections of
curves ~$\Delta$ given by successive applications of
Proposition~\ref{P:algopantsdecomp-main} and removing, for any
subfamily of $\Delta$ of several homotopic curves, all but the
shortest one of them. The number of applications of
Proposition~\ref{P:algopantsdecomp-main} is bounded by the number of
pants in a pants decomposition, which is $2g-2$.

  There remains to bound the length of the closed curves in the pants
  decomposition. A small computation shows that $\ell_k+4n/\ell_k+2
  \leq\ell_{k+1}$ for $C$ large enough and $k\leq 3n$, which holds
  since $k\leq 3g-3\le3n$. Now, $|\partial S^{(1)}| \leq C
  \sqrt{n}=\ell_1$, and applying Proposition 4.2 inductively on
  $S^{(k-1)}$ with $\ell=\ell_{k-1}$ shows that $|\partial S^{(k)}|
  \leq \ell_k= C\sqrt{kn}$.  Therefore, the length of the $k$th closed
  curve of the pants decomposition is at most $C\sqrt{kn}$. The total
  complexity of this algorithm is $O(gn)$ since we applied $O(g)$
  times Proposition~\ref{P:algopantsdecomp-main}.
\end{proof}

Now, onwards to the proof of the main proposition.

\begin{proof}[Proof of Proposition~\ref{P:algopantsdecomp-main}]
  We will only describe how $\Delta$ is computed, since one directly
  obtains $S'$ by cutting along $\Delta$ and discarding the annuli and one
  pair of pants.

  The idea is to \emph{shift} the boundary components simultaneously until
  one boundary component \emph{splits}, or two boundary components
  \emph{merge}.  This is analogous to Morse theory on the surface with the
  function that is the distance to the boundary. In this way, we choose
  the homotopy classes of the curves in $\Delta$, but in order to
  control their length we actually do some backtracking before
  splitting or merging.

  Initially, let $\Gamma=(\gamma^1, \ldots, \gamma^k)$ be (curves
  infinitesimally close to) the boundaries of $S$.  We will shift these
  curves to the right while preserving their simplicity, disjointness, and
  homotopy classes.  We orient each $\gamma^i$ so that it has the surface
  to its right at the start.  In particular, at any time of the algorithm,
  any two curves are to the right of each other.

  \medskip

\begin{figure}
  \centering
  \def\svgwidth{\linewidth}
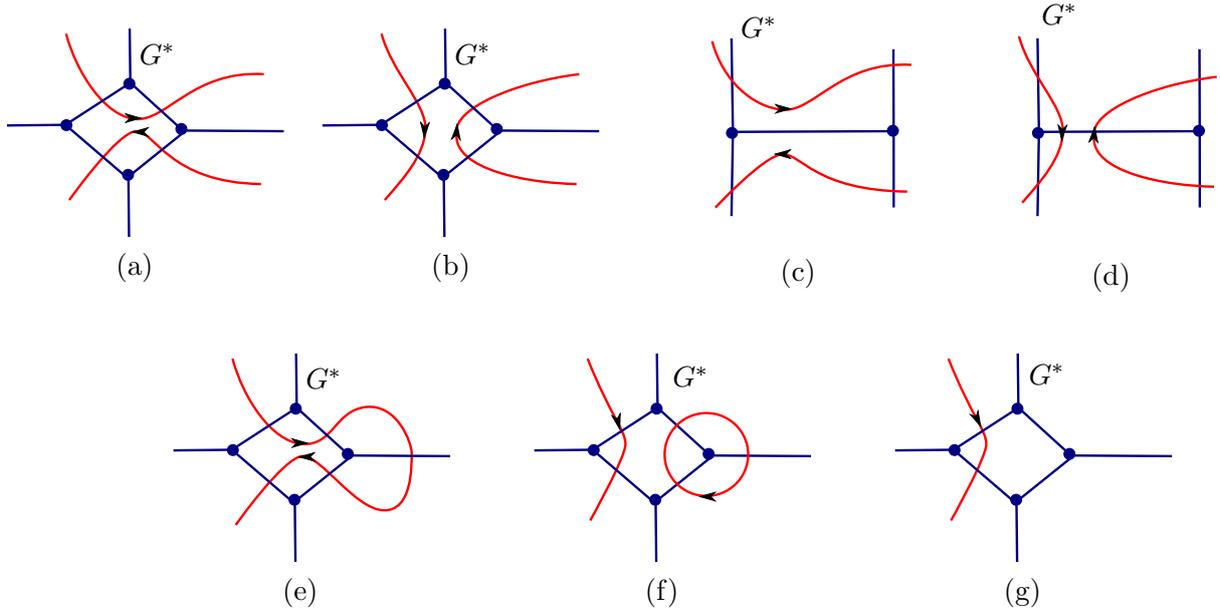%

  \caption{(a) Two tangent pieces of curves lying in the same face. (b) The
    rewiring of these curves. (c) Two tangent pieces of curves lying in
    adjacent faces. (d) The rewiring of these curves. (e) A curve that is
    tangent with itself. (f) Its rewiring. (g) The result after discarding
    the contractible subcurve.}
  \label{fig:pushing}
\end{figure}

  \paragraph*{Shifting phase:}\ 
  The idea of shifting a closed curve $\gamma^i$ one step to the right is
  to push it so that every point of the resulting curve is exactly at
  distance one from the original curve. The shifting phase consists of
  shifting every curve in $\Gamma$ one step to the right, and to reiterate.  During this
  process, curves will collide, which will allow us to build the new curves
  of the pants decomposition.

  A \emph{piece} of a curve in~$\Gamma$ is a maximal subpath inside a face
  of~$G^*$.  We say that two distinct pieces of curves in~$\Gamma$ are
  \emph{tangent} if (i) they are not consecutive pieces along the same
  curve and (ii) there is a path on the surface that starts to the right of
  one piece, arrives to the right of the other, crosses no piece, and
  crosses at most one edge of~$G^*$, see Figures~\ref{fig:pushing}(a, c,
  e).

  Basically, tangencies are the obstacles to shifting the curve to the
  right.  On the other hand, in a tangency, we can \emph{rewire} the curves
  as shown on Figure~\ref{fig:pushing}(b, d, f), by locally exchanging the
  connections between the pieces without changing the orientations of the
  pieces.  Our algorithm needs first to remove all tangencies in~$\Gamma$,
  by repeating the following steps while there exists a tangency:
  \begin{itemize}
  \item If the pieces involved in the tangency belong to the same closed
    curve, then, by the chosen orientation, the rewiring necessarily
    transforms the initial curve into exactly two curves, which we test for
    contractibility.  If one of them is contractible, we discard it
    (Figure~\ref{fig:pushing}(g)) and continue with the other one.
    Otherwise, both are non-contractible; the shifting phase is over, and
    we go to the splitting phase below.
  \item If the tangency involves pieces belonging to different closed
    curves in~$\Gamma$, the rewiring transforms the two curves into a
    single curve; the shifting phase is over, and we go to the merging
    phase below.
  \end{itemize}

  At this step, we removed all tangencies without entering the splitting or
  the merging phase.  Since $G^*$ is trivalent, if $\gamma^i$ were to cross
  consecutively two edges that are incident to the same vertex~$v$ to the
  right of the curve, it would form a tangency with the third edge incident
  to~$v$, a contradiction.  Thus, the local picture is as on
  Figure~\ref{fig:shifting}(a): The edges of~$G^*$ to the right
  of~$\gamma^i$, incident to the faces traversed by~$\gamma^i$, form a
  cycle (the horizontal line in Figure~\ref{fig:shifting}); each edge
  incident to the cycle is either to its left or to its right, and these
  edges are attached to the cycle by distinct vertices; and $\gamma^i$
  crosses exactly those edges of~$G^*$ that are to its left.  We transform
  $\gamma^i$ so that it now crosses exactly those edges that are to the
  right of the cycle, as shown on Figure~\ref{fig:shifting}(b).  The
  absence of tangencies ensures that this still gives disjoint simple
  curves, with the same homotopy classes; of course, this operation may
  create one or several tangencies (in particular, a face of~$G^*$ may now
  be traversed by several pieces).
\begin{figure}
\centering
\def\svgwidth{.6\linewidth}
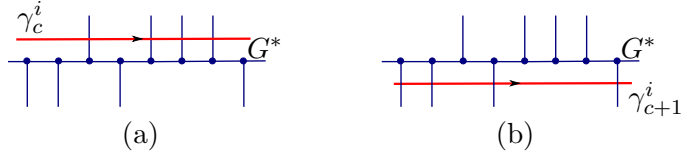%

\caption{Shifting a curve one step to its right.}
\label{fig:shifting}
\end{figure}

When this is done, we repeat the entire shifting phase (again starting with
the tangency detection).  Thus, the shifting phase is repeated over and
over, until we enter the splitting phase or the merging phase below.
Before describing these phases, let us describe some properties that are
satisfied when we exited the shifting phase.  Let $r$ be the integer such
that each curve has been pushed $r$~steps to the right.  For each~$i$,
$1\le i\le k$, and each~$c$, $0\le c\le r$, let $\gamma^i_c$ be the
curve~$\gamma^i$ pushed by $c$~steps.  
Note that by construction, the distance between any point of $\gamma_{c}^i$
and the curve $\gamma_{c-1}^i$ is exactly one. Let $s$ denote the
largest~$c\le r$ such that $\sum_{i=1}^k |\gamma^i_c| \leq \ell$.
(Remember that this is the case for~$c=0$ by hypothesis.)

  \medskip

\begin{figure}
  \centering \def\svgwidth{\linewidth} %
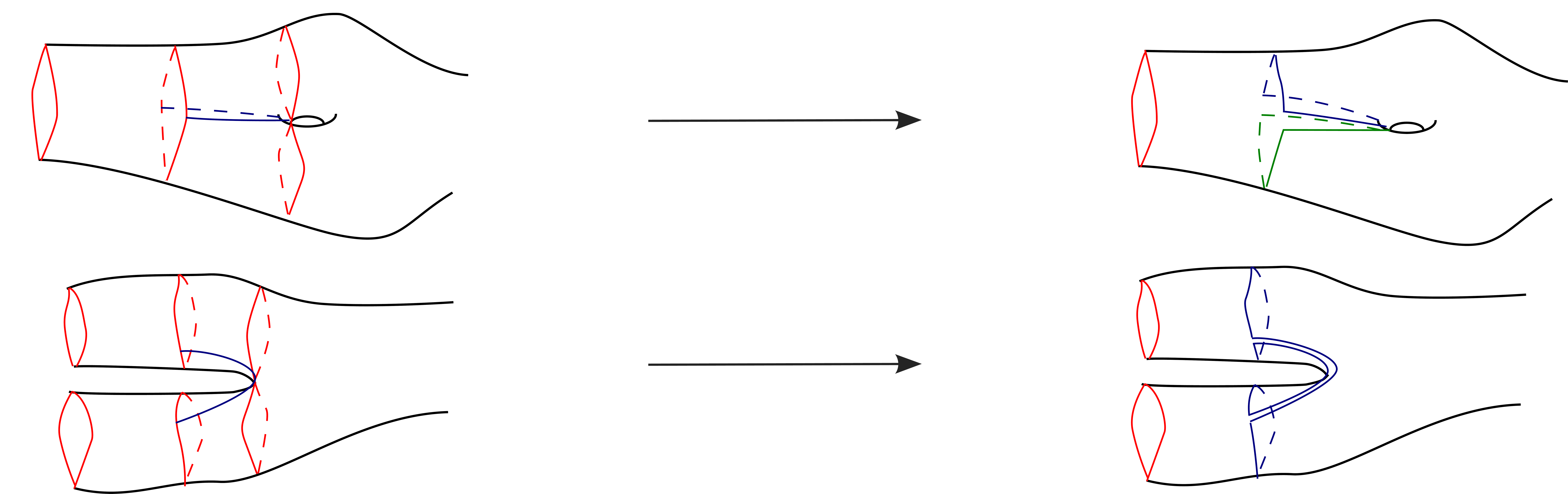%

  \caption{(a) Splitting phase. (b) Merging phase.}
  \label{fig:JoinSplit}
\end{figure}

\paragraph*{Splitting phase:}\ 
We arrived to the splitting phase because two pieces of the same curve
became tangent, and after rewiring, both of the new subcurves are
non-contractible, as is pictured on the top of
Figure~\ref{fig:JoinSplit}. The purpose of the splitting phase is to choose
geometric representatives of curves in these homotopy classes.  For
simplicity, let $\gamma^1$ denote the curve that became tangent with itself
during the shifting phase.  First, for every $i\ne1$, we add $\gamma^i_s$
to the family~$\Delta$.  By assumption, $\gamma^1$ splits into two
non-contractible closed curves $\alpha$ and~$\beta$. Let~$\eta$ be the
shortest path with endpoints on $\gamma^1_s$ going through the splitting
tangency between $\alpha$ and $\beta$.  This path can be computed in linear
time (in the complexity of the portion of the surface swept during the
shifting phase) by backtracking from $\gamma^1_r$ to~$\gamma^1_s$, and
adding pieces of~$\eta$ at every step. The path~$\eta$ cuts~$\gamma^1_s$
into two subpaths $\mu$ and~$\nu$.  We denote by $\delta_1$ the
concatenation of $\mu$ and $\eta$, and by $\delta_2$ the concatenation of
$\nu$ and $\eta$. To finish the splitting phase, we add $\delta_1$ and
$\delta_2$ to the family $\Delta$.

  \medskip

  \paragraph*{Merging phase:}\ 
  We arrived to the merging phase because two distinct shifted curves
  became tangent in the shifting phase (Figure~\ref{fig:JoinSplit},
  bottom); and we rewired them, obtaining a curve homotopic to their
  concatenation.  The purpose of the merging phase is to choose a geometric
  representative in this homotopy class.  For simplicity, let us denote by
  $\gamma^1$ and~$\gamma^2$ two curves that became tangent during the
  shifting phase.  First, for every $i\ne1,2$, we add $\gamma^i_s$ to the
  family~$\Delta$.  Let~$\eta$ be the shortest path from~$\gamma^1_s$
  and~$\gamma^2_s$ (as above, we can compute it in linear time). The
  curve~$\delta$ is defined by the concatenation
  $\eta^{-1}\cdot\gamma^1_s\cdot\eta \cdot\gamma^2_s$.  To finish the
  merging phase, we add $\delta$ to $\Delta$.

  \medskip

  \paragraph*{Analysis:}\ 
  After splitting or merging, we added curves to~$\Delta$ that cut the
  surface into an additional pair of pants, (possibly) some annuli, and the
  remaining surface $S'$.  Observe that we did not add any contractible
  closed curve to $\Delta$; thus, $S'$ has no connected component that is a
  disk.  There remains to prove that the length of the boundary $S'$
  satisfies $|\partial S'| \leq \ell +4n / \ell +2$. The key quantitative
  idea is the way in which the value of~$s$ was chosen: If $s$ was equal
  to~$r$ (perhaps the most natural strategy), the boundary of~$S'$ would
  contain (at least) one curve~$\gamma^i_r$, and we would have no control
  on its length.  On the opposite, if we had chosen $s=0$, we would have no
  control on the lengths of the arcs~$\eta$ involved in the merging or the
  splitting.  The choice of~$s$ gives the right trade-off in-between: the
  lengths of the curves $\gamma_i^s$ are controlled by this threshold,
  while the lengths of the arcs are controlled by the area of the annulus
  between $\gamma^i_s$ and $\gamma^i_r$. We now make this explanation
  precise.

  \smallskip

  \emph{Lengths after the splitting phase.}\quad After a splitting
  phase with the curve~$\gamma^1$, the boundary $\partial S'$ of~$S'$
  consists of all the other curves $\gamma^i_s$ in $\Gamma$ and of
  the two new curves, whose sum of the lengths is bounded by
  $|\gamma^{1}_{s}|+2|\eta|$. Hence $|\partial S'| \leq |\gamma^1_{s}|
  + 2|\eta|+\sum_{i=2}^k |\gamma^i_{s}|$, which is at most
  $\ell+2|\eta|$ by the choice of~$s$.  Furthermore, by construction,
  $|\eta| \leq 2(r-s)+1$, as every step of shifting adds at most 2 to
  the length of $\eta$, and it may cost an additional 1 to cross the last tangency edge.

  \smallskip

  \emph{Lengths after the merging phase.}\quad After a merging phase
  with the curves $\gamma^{1}$ and $\gamma^{2}$, the
  boundary~$\partial S'$ of~$S'$ consists of all the other curves
  $\gamma^i_s$ of $\Gamma$, and of the new curve, whose length is
  bounded by $|\gamma^1_s|+|\gamma^2_s|+2|\eta|$.  Hence similarly,
  $|\partial S'| \le \ell +2|\eta|$.  Furthermore, by construction, we
also have $|\eta| \leq 2(r-s)+1$.
 
  \smallskip

  \emph{Final analysis.}\quad Thus, after either the splitting or the
  merging phase, we proved that $|\partial S'|\le\ell+{4(r-s)+2}$.  To
  conclude the analysis, there only remains to prove that $r-s\le\frac
  {n}{\ell}$.

  Let $c\in\{s,\ldots,r-1\}$.  The curves $\gamma^i_c$
  and~$\gamma^i_{c+1}$ bound an annulus~$K_c^i$.  We claim that the
  number $A(K_c^i)$ of vertices in the interior of this annulus, its
  \emph{area}, is at least $|\gamma^i_{c+1}|$.  This follows from the
  shifting procedure (refer back to
  Figure~\ref{fig:shifting}---remember that $G^*$ is trivalent) and
  from the fact that the contractible closed curves possibly stemming
  from~$\gamma^i_c$ only make the area larger, by definition of
  a tangency.

  For $c\in\{s,\ldots,r-1\}$ and $i\in\{1,\ldots,k\}$, the annuli
  $K_c^i$ have disjoint interiors, so the sum of their areas is at
  most~$n$.  By the above formula, this sum is at least
  $\sum_{j=s}^{r-1}U_{c+1}$, where $U_c=\sum_{i=1}^k |\gamma^i_c|$.
  On the other hand, we have $U_{c+1}\ge\ell$ if $s\le c\le r-1$, by
  definition of~$s$.  Putting all together, we obtain $n\ge
  (r-s)\ell$, so $r-s\le\frac {n}{\ell}$.

  \medskip

  \paragraph*{Complexity:}\ 
  At the start, the complexity of the set of curves is bounded by the
  complexity of $(S,G^*)$, and by construction, during the algorithm, the
  complexity of the set of curves is always linear in~$n$. The complexity
  of the splitting phase or the merging phase is thus also linear
  in~$n$. The complexity of outputting the new surface $(S', G^*_{|S'})$ is
  linear in the complexity $\partial S'$, which is, by construction, also
  linear in~$n$.  To conclude, it suffices to prove that the whole shifting
  phase takes linear time.  We study separately the tangency detection step
  and the contractibility tests.

  \emph{Tangency detection.}  Remember that our curves are stored on
  the cross-metric surface: At each time, we maintain the
  arrangement~$A$ of the overlay of the curves in~$\Gamma$ with~$G^*$.
  On each face~$f$ of~$A$, we store a list~$L(f)$ of pointers to the
  pieces incident to that face and having that face to their right.
  Thus, $f$ contains a tangency if and only if $|L(f)|\ge2$.
  Similarly, if $g$ is a face of~$A$ incident to~$f$ via an edge
  of~$G^*$, the union of $f$ and~$g$ contains a tangency if and only
  if $|L(f)\cup L(g)|\ge3$, or $|L(f)\cup L(g)|=2$ and the two
  corresponding pieces are not consecutive.  These properties can be
  tested in constant time.

  As we push the curves, we update the corresponding lists~$L(f)$.  At the
  start of the shifting, or once the curves have been pushed by one step,
  we first detect the tangencies within the same face~$f$, and deal with
  them, updating the lists $L(f)$.  At this step, there is at most one
  piece per face of~$G^*$.  For every piece of~$\Gamma$, we mark the edges
  incident to the face to the right of that piece; as soon as one edge is
  marked from both sides, and the two corresponding pieces are not
  consecutive, there is a tangency, which we handle immediately.  The
  running time for one tangency detection step is the total complexity of
  the faces that are incident to the curves, and to their right; the sum of
  these complexities is linear in~$n$.  (Note that we only care about the
  part of the surface that is to the right of the curves; the data
  structures involving faces of the remaining part of the surface are
  irrelevant.)

  \emph{Contractibility tests.}  Finally, to perform a contractibility test
  on two subcurves $\alpha$ and $\beta$, we perform a tandem search on the
  surfaces bounded by $\alpha$ and $\beta$, and stop as soon as we find a
  disk. If we find one, the complexity in the tandem search is at most
  twice the complexity of this disk, which is immediately discarded and
  never visited again. If we do not find a disk, the complexity is linear
  in $n$, but the shifting phase is over. Therefore, the total complexity
  of the contractibility tests is linear in the number of vertices swept by
  the shifting phase or in the disks, until the very last contractibility
  test, which takes time linear in $n$. In the end, the shifting phase
  takes time linear in~$n$, which concludes the complexity analysis.
\end{proof}
\section{Shortest Cellular Graphs with Prescribed Combinatorial
  Maps}\label{S:eric}

Guth, Parlier, and Young proved the following result:
\begin{theorem}[{\cite[Theorem~2]{gpy-pdrs-11}}]\label{T:gpy}
  For any $\varepsilon>0$, the following holds with probability tending to
  one as $n$ tends to~$\infty$: A random (trivalent, unweighted)
  cross-metric surface without boundary with $n$~vertices has no pants
  decomposition of length at most $n^{7/6-\varepsilon}$.
\end{theorem}
In this statement, two cross-metric surfaces are regarded as equal if some
self-homeomorphism of the surface maps one to the other.  (Note that
vertices, edges, and faces are unlabeled.)  As a side remark, by a simple
argument, we are actually able to strengthen this result, by replacing, in
the statement above, ``pants decomposition'' by ``genus zero
decomposition''.  We defer the proof of this side result, independent of
the following considerations, to Appendix~\ref{A:genuszero}.

The main purpose of this section is to provide an analogous statement, not
for pants decompositions or genus zero decompositions, but for cut graphs
(or, actually, for arbitrary cellular graphs) with a prescribed
combinatorial map.  We essentially prove that, for any combinatorial
map~$M$ of any cellular graph embedding (in particular, of any cut graph)
of genus~$g$, there exists a (trivalent, unweighted) cross-metric
surface~$S$ with $n$ vertices such that any embedding of~$M$ on~$S$ has
length $\Omega(n^{7/6})$.  We are not able to get this result in full
generality, but are able to prove that it holds for infinitely many values
of~$g$.  On the other hand, the result is stronger since, as in
Theorem~\ref{T:gpy}, it holds ``asymptotically almost surely'' with respect
to the uniform distribution on unweighted trivalent cross-metric surfaces
with given genus and number of vertices.

Let $(S,G^*)$ be a cross metric surface without boundary, and $M$~a
combinatorial map on~$S$.  The \textit{$M$-systole} of $(S,G^*)$ is
the minimum among the lengths of all graphs embedded in~$(S,G^*)$ with
combinatorial map~$M$.  Given $g$ and~$n$, we consider the
set~$\mathcal{S}(g,n)$ of trivalent unweighted cross-metric
surfaces of genus~$g$, without boundary, and with $n$~vertices, where we
regard two cross-metric surfaces as equal if some self-homeomorphism of the
surface maps one to the other.  This refines the model introduced by Gamburd and
Makover~\cite{gm-grrs-02}.  Here is our precise result:
\begin{theorem}\label{T:gpy-cutgraphs}
  Given strictly positive real numbers $p$ and~$\epsilon$, and integers
  $n_0$ and~$g_0$, there exist $n\ge n_0$ and $g\ge g_0$ such that, for any
  combinatorial map~$M$ of a cellular graph embedding with genus~$g$, with
  probability at least $1-p$, a cross-metric surface chosen uniformly at
  random from $\mathcal{S}(g,n)$ has $M$-systole at least
  $n^{7/6-\varepsilon}$.
\end{theorem}
We can obtain a similar result in the case of polyhedral triangulations,
namely, metric spaces obtained by gluing $n$~equilateral Euclidean
triangles with sides of unit length.  We first note that an element of
$\mathcal{S}(g,n)$ naturally corresponds to a polyhedral triangulation by
gluing equilateral triangles of unit side length on the vertices. The
notion of $M$-systole is defined similarly in this setting, and we now
prove that Theorem~\ref{T:gpy-cutgraphs} implies an analogous result for
polyhedral triangulations:

\begin{theorem}\label{T:gpy-cutgraphs-polyhedral}
  Given strictly positive real numbers $p$ and~$\epsilon$, and integers
  $n_0$ and~$g_0$, there exist $n\ge n_0$ and $g\ge g_0$ such that, for any
  combinatorial map~$M$ of a cellular graph embedding with genus~$g$, with
  probability at least $1-p$, a polyhedral triangulation chosen uniformly
  at random from $\mathcal{S}(g,n)$ has $M$-systole at least
  $n^{7/6-\varepsilon}$.
\end{theorem}

\subsection{Proof of Theorem~\ref{T:gpy-cutgraphs}}

\begin{figure}
\centering                                                                
 \begin{tabular}{ccc}
    \includegraphics[width=.25\linewidth]{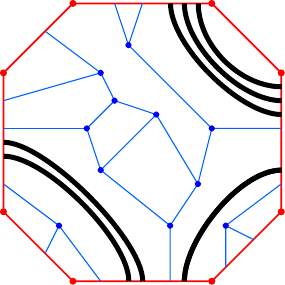} &
    \includegraphics[width=.25\linewidth]{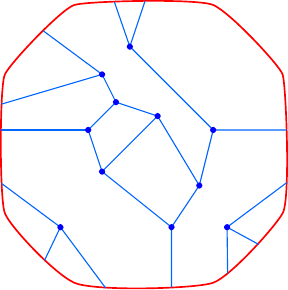}&
    \includegraphics[width=.25\linewidth]{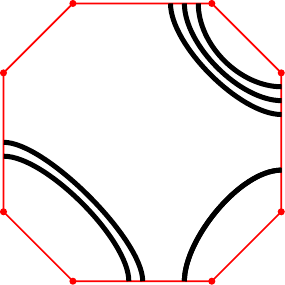}\quad\  \\
    a. & b. & c. \\
  \end{tabular}
  \caption{a. The graph~$H$, obtained after cutting~$\surf$ open
    along~$C$.  The vertices in~$B$ (on the outer face) and the vertices
    of~$G^*$ (not on the outer face) are shown. The chords are in thick
    (black) lines.  b. The graph~$H_1$.  c. The graph~$H_2$.}
  \label{F:count-schemas}
\end{figure}
The general strategy of the proof of Theorem~\ref{T:gpy-cutgraphs} is
inspired by Guth, Parlier and Young~\cite{gpy-pdrs-11}, who proved a
related bound for pants decompositions; however, the details of the method
are rather different.  Our main tool is the following proposition.
\begin{proposition}\label{P:gpy-cutgraphs}
  Given integers $g$, $n$, and~$L$, and a combinatorial map~$M$ of a
  cellular graph embedding of genus~$g$, at most
  \[f(g,n,L)=2^{O(n)}L\left(L/g+1\right)^{12g-9}\] cross-metric surfaces
  in~$\mathcal{S}(g,n)$ have $M$-systole at most~$L$.
\end{proposition}
\begin{proof}
  First, note that it suffices to prove the result for cut graphs with
  minimum degree at least three.  Indeed, one can transform any cellular
  graph embedding into such a cut graph by removing edges, removing
  degree-one vertices with their incident edges, and \emph{dissolving}
  degree-two vertices, namely, removing them and replacing the two incident
  edges with a single one.  For a combinatorial map $M$ with minimum degree
  at least three, Euler's formula and double-counting immediately imply
  that $M$~has at most $4g-2$~vertices and $6g-3$~edges. Given a
  cross-metric surface $(\surf,G^*)$ in~$\mathcal{S}(g,n)$, let $C$ be a
  cut graph of genus~$g$ with combinatorial map $M$ and length at most~$L$.

  Let $H'$ be the graph that is the overlay of $G^*$ and~$C$.
  Cutting~$\surf$ along~$C$ yields a topological disk~$D$, and
  transforms~$H'$ into a connected graph~$H$
  (Figure~\ref{F:count-schemas}(a)) embedded in the plane, where the outer
  face corresponds to the copies of the vertices and edges of the cut
  graph~$C$.  The set~$B$ of vertices of degree two on the outer face
  of~$H$ exactly consists of the copies of the vertices of~$C$; there are
  at most~$12g-6$ of these. A \emph{side} of~$H$ is a path on the boundary
  of~$D$ that joins two consecutive points in~$B$.

  Given the combinatorial map of~$H$ in the plane, we can (almost) recover
  the combinatorial maps corresponding to $H'$ and to $(\surf,G^*)$.
  Indeed, the set~$B$ of vertices of degree two on the outer face of~$H$
  determines the sides of~$D$.  The correspondence between each side of~$D$
  and each edge of the combinatorial map~$M$ is completely determined once
  we are given the correspondence between a single half-edge on the outer
  face of~$H$ and a half-edge of~$M$; in turn, this determines the whole
  gluing of the sides of~$H$ and completely reconstructs $H'$ with
  $C$~distinguished.  Finally, to obtain $G^*$, we just ``erase''~$C$.
  Therefore, one can reconstruct the combinatorial map corresponding to the
  overlay~$H'$ of~$G^*$ and~$C$, just by distinguishing one of the $O(L)$
  half-edges on the outer face of~$H$.

  \begin{figure}\centering
    \begin{tabular}{cc}
    \includegraphics[width=.25\linewidth]{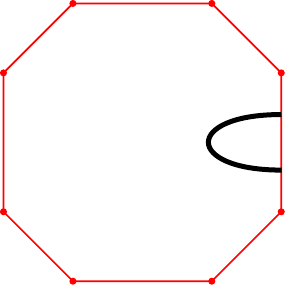} &
    \includegraphics[width=.25\linewidth]{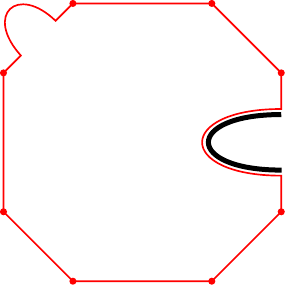}\\
    a. & b.\\
  \end{tabular}
    \caption{The exchange argument to prove~(i).}
    \label{F:count-schemas-detail-1}
  \end{figure}

  A \emph{chord} of~$H$ is an edge of~$H$ that is not incident to the outer
  face but connects to vertices incident to the outer face.  Two chords are
  \emph{parallel} if their endpoints lie on the same pair of sides of~$D$.
  We claim that we can assume the following:
  \begin{enumerate}
  \item[(i)] no chord has its endpoints on the same side of~$H$
    (Figure~\ref{F:count-schemas-detail-1}(a) shows an example not
    satisfying this property);
  \end{enumerate}
  and that (at least) one of the two following conditions holds:
  \begin{enumerate}
  \item[(ii)] the subgraph of~$H$ between any two parallel chords only
    consists of other parallel chords
    (Figure~\ref{F:count-schemas-detail-2}(a) shows an example not
    satisfying this property), or
  \item[(ii')] there are two parallel chords such that the subgraph of~$H$
    between them contains all the interior vertices of~$H$.
  \end{enumerate}
  Indeed, without loss of generality, we can assume that our cut graph~$C$
  has minimum length among all cut graphs of~$(\surf,G^*)$ with
  combinatorial map~$M$.  If a chord violates~(i), one could shorten the
  cut graph by sliding a part of the cut graph over the chord
  (Figure~\ref{F:count-schemas-detail-1}), which is a contradiction.  

  \begin{figure}\centering
    \inkfrag{c1}{$c_1$} \inkfrag{c2}{$c_2$} \inkfrag{p1}{$p_1$}
    \inkfrag{p2}{$p_2$}
    \begin{tabular}[c]{ccc}
    \raisebox{-6.2mm}{%
  \begin{inkfragenv}\def\svgwidth{.25\linewidth}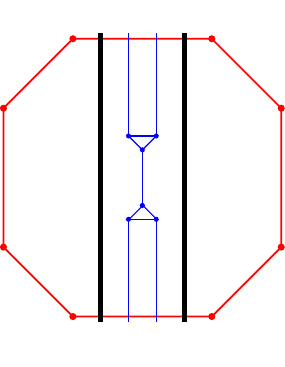\end{inkfragenv}%
} &
  \begin{inkfragenv}\def\svgwidth{.25\linewidth}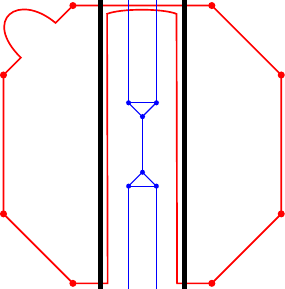\end{inkfragenv}%
 &
    \raisebox{-6.7mm}{%
  \begin{inkfragenv}\def\svgwidth{.31\linewidth}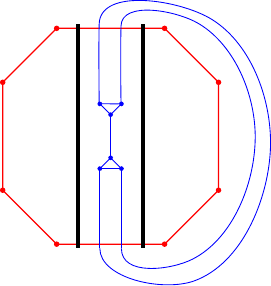\end{inkfragenv}%
} \\
    a. & b. & c.\\
  \end{tabular}
  \caption{a.: Two chords violating~(ii).  b.: The exchange argument, in
    case $p_1$ and~$p_2$ have different perturbed lengths.  c.: A schematic
    view of the situation, in case $p_1$ and~$p_2$ have the same perturbed
    length.}
  
  \label{F:count-schemas-detail-2}
  \end{figure}
  
  For (ii) and (ii'), the basic idea is to use a similar exchange argument
  as to prove~(i), but we need a perturbation argument as
  well. Specifically, let us temporarily perturb the crossing weights of
  the edges of~$G^*$ as follows: The weight of each edge~$e$ of~$G^*$
  becomes~$1+w_e$, where the $w_e$'s are real numbers that are linearly
  independent over~$\Q$ (e.g., independent and identically distributed
  random) and strictly between $0$ and~$1/L$.  Let $C$ be a shortest
  embedded graph with combinatorial map $M$ under this perturbed metric.

  It is easy to see that $C$ is also a shortest embedding with
  combinatorial map $M$ under the unweighted metric: Indeed, two cut graphs
  $C_1$ and~$C_2$ with respective (integer) lengths $\ell_1<\ell_2\le L$ in
  the unweighted metric have respective lengths $\ell'_1<\ell'_2$ in the
  perturbed metric, since the perturbation increases the length of each
  edge by less than~$1/L$.

  We claim that either (ii) or~(ii')~holds for this choice of~$C$.  Assume
  that (ii)~does not hold; we prove that (ii')~holds.  So the region~$R$
  of~$D$ between two parallel chords $c_1$ and~$c_2$ of~$D$ contains
  internal vertices; without loss of generality (by~(i)), assume that the
  region~$R$ contains no other chord in its interior.  Let $p_1$ and~$p_2$
  be the two subpaths of the cut graph on the boundary of~$R$.  If $p_1$
  and~$p_2$ have different lengths under the perturbed metric, e.g., $p_1$
  is shorter, then we can push the part of~$p_2$ to let it run along~$p_1$
  and shorten the cut graph (Figure~\ref{F:count-schemas-detail-2}(b)),
  which is a contradiction.  Therefore, $p_1$ and~$p_2$ have the same
  length under the perturbed metric, which implies that they cross exactly
  the same set~$E$ of edges of~$G^*$, since the weights are linearly
  independent over~$\Q$.  (We exclude from~$E$ the edges on the endpoints
  of~$p_1$ and~$p_2$.)  Since none of the edges in~$E$ are chords, all the
  endpoints of the edges in~$E$ belong to ~$R$
  (Figure~\ref{F:count-schemas-detail-2}(c)), which implies~(ii') by
  connectivity of~$G^*$. This concludes the proof of the claim.

  We now estimate the number of possible combinatorial maps for~$H$, by
  ``splitting'' it into two connected plane graphs $H_1$ and $H_2$,
  estimating all possibilities of choosing each of these graphs, and
  estimating the number of ways to combine them.

  Let $H_1$ be the graph (see Figure~\ref{F:count-schemas}(b)) obtained
  from~$H$ by removing all chords and dissolving all degree-two vertices
  (which are either in~$B$ or endpoints of a chord).  $H_1$ is connected,
  trivalent, and has at most~$n$ vertices not incident to the outer face,
  so $O(n)$ vertices in total.  By a classical calculation (see for
  example~\cite[Lemma~4]{gpy-pdrs-11}), there are thus $2^{O(n)}$ possible
  choices for the combinatorial map of this planar trivalent graph~$H_1$.

  On the other hand, let $H_2$ be the graph (see
  Figure~\ref{F:count-schemas}(c)) obtained from~$H$ by removing internal
  vertices together with their incident edges and dissolving all degree-two
  vertices not in~$B$.  Since the chords are non-crossing and connect
  distinct sides of~$D$, the pairs of sides connected by at least one chord
  form a subset of a triangulation of the polygon having one vertex per
  side of~$D$.  To describe~$H_2$, it therefore suffices to describe a
  triangulation of this polygon with at most $12g-6$ edges, which makes
  $2^{O(g)}=2^{O(n)}$ possibilities, and to describe, for each of the
  $12g-9$ edges of the triangulation, the number of parallel chords
  connecting the corresponding pair of sides.  Since there are at most
  $L$~chords, the number of possibilities for these numbers equals
  $\Setbar{(x_1,\ldots,x_{12g-9})}{x_i\ge0, \sum_ix_i\le L}$, which is the
  number of weak compositions of $L$ into $12g-8$ parts, namely
  \[{L+12g-9\choose 12g-9} \leq
  \left(\frac{e(L+12g-9)}{12g-9}\right)^{12g-9}=O\left((L/g+1)^{12g-9}\right)\times
  2^{O(n)},\]%
  the inequality being standard (or following from Stirling's formula).

  Finally, in how many ways can we combine given $H_1$ and~$H_2$ to
  form~$H$?  Let us first assume that (ii)~holds; the parallel chords
  joining the same pair of sides are consecutive, so choosing the position
  of a single chord fixes the position of the other chords parallel to it.
  Therefore, given $H_1$, we need to count in how many ways we can insert
  the $O(g)$ vertices of~$B$ on~$H_2$ into~$H_1$, and similarly the $O(g)$
  intervals where endpoints of chords can occur, respecting the cyclic
  ordering.  After choosing the position of a distinguished vertex
  of~$H_2$, we have to choose $O(g)$ positions on the edges of the boundary
  of~$H_1$, possibly with repetitions, which leaves us with
  $\binom{O(n+g)}{O(g)}\le2^{O(n+g)}=2^{O(n)}$ possibilities.  In
  case~(ii') holds, a very similar argument gives the same result. 

  The claimed bound follows by multiplying the number of all possible
  choices above: there are $O(L)$ choices for the distinguished half-edge
  of the outer face of~$H$, $2^{O(n)}$ choices for~$H_1$,
  $O\left((L/g+1)^{12g-9}\right)\times 2^{O(n)}$ choices
  for~$H_2$, and $2^{O(n)}$ possibilities for combining $H_1$ and~$H_2$.
\end{proof}

\begin{proof}[Proof of Theorem~\ref{T:gpy-cutgraphs}]
  Let $g_0,n_0,p,\varepsilon$ be as indicated.  Euler's formula implies
  that a cross-metric surface with $n$~vertices has genus $g\le(n+2)/4$.
  We now show that, if $n$~is large enough, \[\sum_{g=g_0}^{(n+2)/4}
  f(g,n,n^{7/6-\varepsilon})\le n^{(1-\varepsilon)n/2}.~(*)\] Indeed, by
  Proposition~\ref{P:gpy-cutgraphs} we have \[f(g,n,n^{7/6-\varepsilon}) \le
  2^{C_0n}\left(n^{7/6-\varepsilon}/g+1\right)^{12g-9}\] for some
  constant~$C_0$.  We need to sum up these terms from $g=g_0$ to~$(n+2)/4$.
  For $n$~large enough, the largest term in this sum is for $g=(n+2)/4$.
  Thus the desired sum is bounded from above by
  \[n2^{C_0n} \left(4n^{1/6-\varepsilon}+1\right)^{12(n+2)/4-9},\] which is
  at most $2^{C_1n} n^{(1/6-\varepsilon) 3n}$ (for $n$ large enough, for
  some constant~$C_1$), which in turn is at most $n^{(1-\varepsilon)n/2}$
  for $n$~large enough.

  Furthermore, let $h(g,n)=|\mathcal{S}(g,n)|$ be the number of (connected)
  cross-metric surfaces with genus~$g$ and $n$~vertices.  We have
  $\sum_{g=0}^{(n+2)/4}h(g,n)\ge e^{Cn}n^{n/2}$ if $n$~is large enough and
  even, for some absolute constant~$C$; this is probably folklore, and we
  provide a proof, deferred to Lemma~\ref{L:count-connected}.  But, if
  $g$~is fixed, $h(g,n)=O(e^{C'n})$ for some constant
  $C'$~\cite[Lemma~4]{gpy-pdrs-11}.  Thus, since $g_0$ is fixed, there is a
  constant~$C''$ such that, for $n$ large enough and even,
  $\sum_{g=g_0}^{(n+2)/4}h(g,n)\ge e^{C''n} n^{n/2}$~(**).

  Choose any (even)~$n\ge n_0$ such that $n^{-\varepsilon n/2}e^{-C''n}\le
  p$ and such that (*) and~(**) hold.  Thus, we have
  \[\sum_{g=g_0}^{(n+2)/4}f(g,n,n^{7/6-\varepsilon}) \leq
  p \sum_{g=g_0}^{(n+2)/4}h(g,n),\] %
  which implies that for some~$g\ge g_0$,
  \[f(g,n,n^{7/6-\varepsilon})/h(g,n) \le p\] and the denominator is
  non-zero.  In other words, among all $h(g,n)$~cross-metric surfaces with
  genus~$g$ and $n$~vertices, for any combinatorial map~$M$ of a cellular
  graph embedding of genus~$g$, a fraction at most~$p$ of these surfaces
  have an embedding of~$M$ with length at most~$n^{7/6-\varepsilon}$.
\end{proof}

We remark that a tighter estimate on the number~$h(g,n)$ of triangulations
with $n$~triangles of a surface of genus~$g$ could lead to the same result
\emph{for any large enough~$g$}, instead of \emph{for infinitely many
  values of~$g$}.

To conclude the proof, there remains to prove the bound on the number of
connected surfaces.

\begin{lemma}\label{L:count-connected}
  The number of (trivalent, unweighted) \emph{connected} cross-metric
  surfaces with $n$~vertices without boundary is, for $n$~even large
  enough, at least $e^{Cn}n^{n/2}$ for some absolute constant~$C$.
\end{lemma}
\begin{proof}
  Let $G_n$ be the set of simple unlabelled trivalent graphs with
  $n$~vertices.  Let $G'_n$ be the set of graphs in~$G_n$ that are
  connected.  Let $S'_n$ be the number of connected cross-metric surfaces
  with $n$~vertices; we want a lower bound on~$|S'_n|$.  Below we
  implicitly assume $n$ to be even, for otherwise these sets are empty.

  We have $|S'_n|\ge|G'_n|$, because every graph in~$G'_n$ leads to a
  connected cross-metric surface, by cellularly embedding the graph
  arbitrarily, and these cross-metric surfaces are distinct, because they
  have distinct vertex-edge graphs.

  Moreover, $|G'_n|/|G_n|$ tends to one as $n$ goes to infinity, because
  the proportion of 3-connected graphs in the set of simple unlabelled
  trivalent graphs with $n$ vertices goes to one as $n$ goes to
  infinity~\cite[p.~338]{mw-argss-84}.  (Actually, except for this
  argument, our proof is heavily inspired by Guth et al.~\cite[Lemmas 1
  and~3]{gpy-pdrs-11}.)

  The number of simple \emph{labelled} trivalent graphs with $n=2k$
  vertices is, as $n$ goes to infinity, equivalent to $\frac{(6k)!}{(3k)!}
  288^ke^2$~\cite{r-sepgt-58}.  The expected number of automorphisms of
  these graphs tends to one as $n$ goes to
  infinity~\cite[Corollary~3.8]{mw-argss-84}, which implies that $|G_n|$ is
  equivalent to $\frac{(6k)!}{(3k)!(2k)!}\* 288^ke^2$, which is at least
  $e^{Cn}n^{n/2}$ for some absolute constant $C$.  The previous paragraphs
  imply that $|S'_n|$ is asymptotically at least as large, as desired.
\end{proof}

\subsection{Proof of Theorem~\ref{T:gpy-cutgraphs-polyhedral}}

We now show that the result just proved, Theorem~\ref{T:gpy-cutgraphs},
implies the polyhedral variant, Theorem~\ref{T:gpy-cutgraphs-polyhedral}:

\begin{figure}                                                                
  \inkfrag{a}{a.}
  \inkfrag{b}{b.}
  \inkfrag{c}{c.}
  \inkfrag{d}{d.}
  \inkfrag{e}{e.}
  \inkfrag{f}{f.}
  \centerline{%
  \begin{inkfragenv}\def\svgwidth{\linewidth}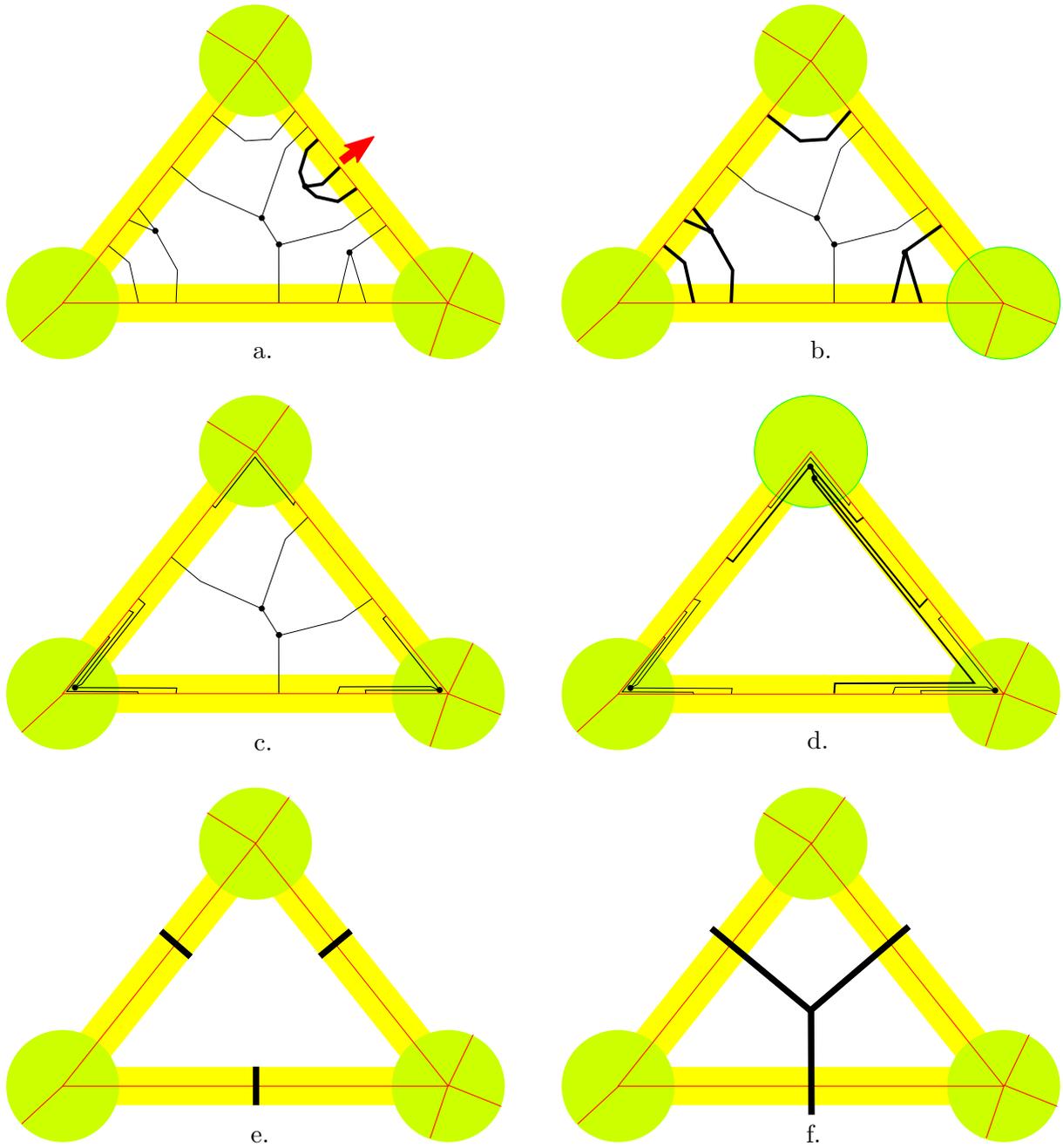\end{inkfragenv}%
}
  \caption{Illustration of the proof of
    Theorem~\ref{T:gpy-cutgraphs-polyhedral}. a.: The disks and strips
    inside one triangle of~$G$, and the part of the cut graph~$C$ inside
    the triangle.  b.: A piece with side number one is pushed across the
    side of the triangle.  c.: The pieces with side number two are pushed
    to the disks and strips.  d.: The piece with side number three is
    pushed to the disks and strips.  e.: The paths $P_s$.  f.: The
    cross-metric surface.}
  \label{F:gpy-cutgraphs-polyhedral}
\end{figure}
\begin{proof}
  As in the proof of Theorem~\ref{T:gpy-cutgraphs}, it suffices to prove
  the result for maps~$M$ that are cut graphs with minimum degree three,
  which have at most $4g-2$~vertices and $6g-3$~edges.  Let $G$ be the
  vertex-edge graph of a polyhedral triangulation on a surface~$S$ with
  genus~$g$.  Assume that $C$ is a graph with combinatorial map~$M$ and of
  length at most $n^{7/6-\varepsilon}$ on that polyhedral surface.  We
  prove that some cut graph with combinatorial map~$M$ has length
  $O(n^{7/6-\varepsilon})$ in the dual cross-metric surface~$(\surf,G^*)$.
  Since, by Theorem~\ref{T:gpy-cutgraphs}, the proportion of such surfaces
  is arbitrarily small, this implies the theorem.

  Without loss of generality, we assume that~$C$ is piecewise-linear, and
  in general position with respect to~$G$.  We consider a tubular
  neighborhood of~$G$ (Figure~\ref{F:gpy-cutgraphs-polyhedral}(a)),
  obtained by first building a small \emph{disk} around each vertex of~$G$,
  and then building a rectangular \emph{strip} containing each part of edge
  not covered by a disk.  The disks are pairwise disjoint, the strips are
  pairwise disjoint, and each strip intersects only the disks corresponding
  to the incident vertices of the corresponding edge, along paths.  We
  push~$C$ into the disks and strips as follows.  A \emph{piece} of~$C$ in
  a triangle~$T$ is a maximal connected part of~$C$ that lies in~$T$; the
  \emph{side number} of a piece is the number of sides of~$T$ it touches.

  First, consider all the pieces with side number one.  By an ambient
  isotopy, we can push these pieces across the side of the triangle they
  touch without increasing their length.  So we can assume that no piece
  has side number one in any triangle.

  Next, we deal with the pieces with side number two.  By an ambient
  isotopy of the triangle fixing its boundary, we push all such pieces into
  the strips of the two sides of the triangle, putting the vertices in the
  disk touching the two strips
  (Figure~\ref{F:gpy-cutgraphs-polyhedral}(b--c)).  Elementary geometry
  implies that this at most doubles the length of the pieces containing no
  vertex of~$C$, and it increases the length of the pieces with a vertex by
  an additional term that is linear in the number of edges incident to the
  vertices of the piece.  Since $C$ has $O(g)=O(n)$ edges, the length of
  the modified cut graph is still $O(n^{7/6-\varepsilon})$.

  Finally, there exists at most one piece with side number three lying in
  each triangle.  We can push that piece as well to the three strips of the
  sides of the triangle, pushing all vertices of that piece to one of the
  disks, chosen arbitrarily (Figure~\ref{F:gpy-cutgraphs-polyhedral}(d));
  this operation increases the length of~$C$ by an additional term that is
  at most the number of edges of the piece.  As before, this additional
  increase in length is $O(g)=O(n)$.

  So, we have obtained an isotopic cut graph~$C'$, whose length is still
  $O(n^{7/6-\varepsilon})$, with the property that the vertices of~$C'$ lie
  in the disks and the edges of~$C'$ lie in the union of the disks and the
  strips.  For each strip~$s$, draw a shortest path~$P_s$, with endpoints
  on its boundary, separating the two incident disks
  (Figure~\ref{F:gpy-cutgraphs-polyhedral}(e)).  If a portion of~$C'$
  inside~$s$ crosses~$P_s$ more than once, it forms a bigon with~$P_s$; by
  flipping innermost bigons, without increasing the length of~$C'$, we can
  assume that each portion of~$C'$ inside~$s$ crosses~$P_s$ at most once.

  Now we extend the paths~$P_s$ to form the graph~$G^*$
  (Figure~\ref{F:gpy-cutgraphs-polyhedral}(f)).  By the paragraph above,
  each crossing of a path~$P_s$ corresponds to a portion of a path of~$C'$
  that crosses the strip containing~$P_s$, and thus has length at least
  $1-\delta$, for $\delta>0$ arbitrarily close to zero (the size of the
  disks and strips are chosen according to~$\delta$).  Therefore, the
  length of~$C'$ on the cross-metric surface $(S,G^*)$ is at most
  $(1-\delta)$ times that of the length of~$C'$ on the polyhedral
  triangulated surface, and thus $O(n^{7/6-\varepsilon})$.
\end{proof}

An interesting question would be to determine whether there exists an
analog of Theorem~\ref{T:gpy-cutgraphs} when we are not given the
embedding of~$M$, but only its abstract graph.  More generally, let
$S$ and $M$ be two graphs with $n$ vertices that are cellularly
embeddable on a surface of genus~$g$; are there cellular embeddings of
$S$ and~$M$ on this surface such that the graphs cross only $O(n)$
times?


\section*{Acknowledgements}
  We would like to thank Jean-Daniel Boissonnat, Ramsay Dyer, and Arijit
  Ghosh for pointing out and discussing with us their results on Voronoi
  diagrams of Riemannian surfaces~\cite{dzm-ssivd-08} and manifolds.  We
  are grateful to the anonymous referees for their careful reading of the
  manuscript, which allowed to correct several problems and to improve the
  presentation significantly, and for pointing out Kowalick's
  thesis~\cite{k-dsi-13}.

\bibliographystyle{plain}

\appendix\normalsize

\section{Discrete Systolic Inequalities in Higher Dimensions }\label{A:hd}

In this appendix, we show that the proofs from Section~\ref{S:alfredo}
extend almost verbatim to higher dimensions. In the following discussion
$(M,T)$ will be a triangulated $d$-manifold.%
\footnote{E.g., $(M,T)$ is a simplicial complex whose underlying space is a
  $d$-manifold.  However, we can allow more general triangulations obtained
  from gluing $d$-simplices, in which, after gluing, some faces (e.g.,
  vertices) of the same $d$-simplex are identified.}
We will denote by $f_d(T)$ the number of $d$-dimensional simplices of~$T$,
and by $f_0(T)$ the number of vertices. The main difference with the
two-dimensional case is that while for surfaces, discrete systolic
inequalities in terms of $f_0$ and in terms of $f_d$ are easily seen to be
equivalent (by Euler's formula and double counting), in higher dimensions
the situation is more complicated.

We consider the supremal values of the functionals $\frac{\sys^d}{f_d}$ and
$\frac{\sys^d}{f_0}$, where $\sys$ denotes the length of a shortest closed
curve in the $1$-skeleton of~$(M,T)$ that is non-contractible on the
manifold~$M$. In particular we focus on when these quantities are bounded
from above. As we surveyed in the introduction, the two-dimensional case of
this problem has been studied by topological graph theorists and
computational topologists; however, as far as we know, it has never been
considered in dimension higher than two in the past. We report the results
and open problems that we can derive by generalizing our techniques for
surfaces.

\subsection{From Continuous to Discrete Systolic Inequalities}

To infer discrete systolic inequalities from the Riemannian ones, the
obvious approach is, as before, to start with a triangulated manifold
$(M,T)$ and to endow~$M$ with a metric $m_T$ by deciding that each simplex
of~$T$ is a regular Euclidean simplex of volume one.  (Since the simplices
are regular, we glue them by facewise isometries.)  Hence, length and
volume are naturally defined via the restriction to each Euclidean simplex.
Following Gromov~\cite{g-frm-83}, we will call such a metric a
\emph{piecewise Riemannian metric}.  Unlike the 2-dimensional case,
however, foundational work of Kervaire~\cite{k-mwdna-60} shows that in
higher dimensions such a triangulated manifold is not always smoothable.
(We will show how to circumvent this difficulty below.)

\begin{theorem}\label{T:conttodiscr-hd}
  There exists a constant $C_d$, such that for every triangulated compact
  manifold $(M,T)$ without boundary of dimension $d$, there exists a piecewise Riemannian
  metric~$m$ on~$M$ with volume~$f_d(T)$ such that for every closed curve
  $\gamma$ in $M$, there exists a homotopic closed curve $\gamma'$ on the
  $1$-skeleton~$G$ of~$T$ with
  \[|\gamma'|_{G}\leq C_d |\gamma|_m.\]
\end{theorem}

The proof works inductively, pushing curves from the $i$-dimensional
skeleton to the ($i-1$)-dimensional one. We start with the following lemma.

\begin{lemma}\label{L:pushing-hd}
  Let $\Delta$ be an $i$-dimensional regular simplex, endowed with the
  Euclidean metric.  There exists an absolute constant $C'_i$ such that,
  for each arc~$\gamma$ properly embedded in $\Delta$ with endpoints in
  $\partial \Delta$, there exists an arc $\gamma'$ embedded on $\partial
  \Delta$, with the same endpoints as $\gamma$, such that $|\gamma'| \leq
  C'_i|\gamma|$.
\end{lemma}

\begin{proof}[Proof of Lemma~\ref{L:pushing-hd}]
  Since the statement of the lemma is invariant by scaling all the
  distances, we can assume that $\Delta$ is the regular $i$-simplex whose
  circumscribing sphere bounds the unit ball~$B$ in~$\R^i$.  Let us first
  consider the bijection~$\varphi$ that maps~$\Delta$ to~$B$ by radial
  projection (such that the restriction of~$\varphi$ to any ray from the
  origin is a linear function). It is not hard to see that there is a constant $C''_i$ such that, for any
  arc~$\gamma$ in~$\Delta$, we have
  $|\gamma|/{C''_i}\le|\varphi(\gamma)|\le C''_i|\gamma|$ (one can compute the optimal $C''_i$ by writing the map in hyperspherical coordinates and computing the differential).

 Therefore it suffices to prove the lemma for the unit ball~$B$
  instead of the regular simplex~$\Delta$.  Let $\beta$ be an arc embedded
  in~$B$.  Let $\beta'$ be a shortest geodesic arc on~$\partial B$ with the
  same endpoints as~$\beta$.  Then we have $|\beta'|\le\frac\pi2|\beta|$,
  which proves the result.
\end{proof}

\begin{proof}[Proof of Theorem~\ref{T:conttodiscr-hd}]
  As we mentioned before, we endow $M$ with the piecewise Riemannian metric
  obtained by endowing each simplex of dimension~$d$ with the geometry of
  the regular Euclidean simplex of volume~$1$.  Then, using
  Lemma~\ref{L:pushing-hd}, for every arc $A$ of $\gamma$ in every
  $d$-simplex, we push $A$ to the ($d-1$)-skeleton of~$(M,T)$, and we
  repeat this procedure inductively until $\gamma$ is embedded in the
  $1$-skeleton. In the end, the length of $\gamma'$ has increased by at
  most a multiplicative factor that depends only on~$d$.
\end{proof}

The Riemannian systolic inequality in higher dimensions is now stated in
the following theorem.

\begin{theorem}[Gromov~\cite{g-frm-83}]\label{T:gromov}
  For every $d$, there is a constant $C_d$ such that, for any Riemannian
  metric~$m$ on any essential compact $d$-manifold~$M$ without boundary, there exists a
  non-contractible closed curve of length at most
  $C_d\vol(m)^{1/d}$.
\end{theorem}
For a definition of essential manifold, see~\cite{g-frm-83}. The prime
examples of essential manifolds are the so-called aspherical manifolds,
which are the manifolds whose universal cover is contractible. These
include for example the $d$-dimensional torus for every~$d$, or manifolds
that accept a hyperbolic metric and, more generally, manifolds that
are locally CAT(0). In particular, all
surfaces except the $2$-sphere and the projective plane are aspherical. On the other hand,
real projective spaces and lens spaces are examples of essential manifolds that are non aspherical.

Theorem~\ref{T:gromov} also holds for
piecewise Riemannian metrics. Indeed, its proof revolves around two key
inequalities: the filling radius-volume inequality and a systole-filling
radius inequality. The former relies on a coarea formula which holds for
piecewise Riemannian metrics (see~\cite[Lemma~4.2b]{g-frm-83}), and the
proof of the latter uses no smoothness property either, see~\cite[p.~9
and 10]{g-frm-83}.  As a corollary of this refinement to piecewise
Riemannian metrics and of our Theorem~\ref{T:conttodiscr-hd}, we obtain the
following result relating the length of systoles and the number of facets.

\begin{corollary}
  Let $(M,T)$ be a triangulated essential compact $d$-manifold without boundary. Then, for
  some constant $c_d$ depending only on $d$, some non-contractible closed
  curve in the $1$-skeleton of $T$ has length at most
  $c_df_d(T)^{1/d}$.
\end{corollary}
  
\subsection{From Discrete to Continuous Systolic Inequalities}

We now turn our attention to the other direction, namely, transforming a
discrete systolic inequality into a continuous one.

\begin{theorem}\label{T:discrtocont-hd}
  Let $M$ be a compact Riemannian manifold of dimension $d$ and volume $V$ without boundary,
  and let $\delta>0$.  For infinitely many values of~$f_0$, there exists a
  triangulation $(M,T)$ of~$M$ with $f_0$~vertices, such that every closed
  curve~$\gamma$ in the $1$-skeleton $G$ of~$M$ satisfies
  \[|\gamma|_m\leq (1+\delta)\frac{10}{\sqrt{\pi}}\Gamma(d/2+1)^{1/d}
  \left(\frac{V}{f_0} \right)^{1/d} |\gamma|_G.\]
\end{theorem}
\noindent (Here, $\Gamma$ is the usual Gamma function.)  The proof follows
the same idea as the proof of Theorem~\ref{T:discrtocont}. We start with
the centers of a maximal set of disjoint balls of radius $\varepsilon/2$ in
$M$ and want to compute the Delaunay triangulation associated to it, with
the hope that if $\epsilon$ is small enough, we will obtain a triangulation
of~$M$. However, Delaunay complexes behave differently in higher
dimensions, and this hope turns out to be false in many cases. We rely
instead on a recent work reported by Boissonnat, Dyer and
Ghosh~\cite{bdg-dtm-13} who devised the correct perturbation scheme to
triangulate a manifold using a Delaunay complex. We will use the following
theorem.

\begin{theorem}\label{T:BDG}
Let $M$ be a compact Riemannian manifold. For a small enough $\varepsilon$, there
exists a point set $P \subseteq M$ such that 
\begin{enumerate}
\item[(i)] For every $x \in M$, there exists $p \in P$ such that $|x-p|_m
  \leq \varepsilon$.
\item[(ii)] For every pair $p \neq p' \in P$, $|p-p'|_m \geq
  2\varepsilon/5$.
\item[(iii)] The Delaunay complex of $P$ is a triangulation of $M$.
\end{enumerate}
\end{theorem}

For completeness, we sketch how to infer this theorem from the
paper~\cite{bdg-dtm-13}.

\begin{proof}[Proof of Theorem~\ref{T:BDG}]
  We say that a set of points $P \subseteq M$ is $\varepsilon$-dense if
  $d(x,P) < \varepsilon$ for $x \in M$, $\mu_0 \varepsilon$-separated if
  $d(p,q)\geq \mu_0 \varepsilon$ for all distinct $p,q \in P$, and is a
  $(\mu_0,\varepsilon)$-net if it is $\varepsilon$-dense and $\mu_0
  \varepsilon$-separated.

  Taking $\mu'_0=1$ and $\varepsilon'$ small enough, we start with a
  $(\mu_0', \varepsilon')$-net in $M$, which can be obtained for example by
  taking the centers of a maximal set of disjoint balls of radius
  $\varepsilon'/2$. Now, the extended algorithm of Boissonnat et
  al.~\cite{bdg-dtm-13} outputs a $(\mu_0, \varepsilon)$-net with
  $\varepsilon \leq 5\varepsilon'/4$ and $\mu_0 \geq 2\mu_0'/5$, which will
  be our point set $P$.  These conditions correspond to items (i) and~(ii)
  in our theorem. For $\varepsilon$ small enough, all the hypotheses of
  their main theorem are fulfilled, and therefore we obtain that the
  Delaunay complex of the $(\mu_0, \varepsilon)$-net is a triangulation
  of~$M$, which is our item (iii).
\end{proof}

The proof of Theorem~\ref{T:discrtocont-hd} now follows the same lines as
in the $2$-dimensional case.

\begin{proof}[Proof of Theorem~\ref{T:discrtocont-hd}]
  Let $\varepsilon>0$ be a constant. Following Theorem~\ref{T:BDG}, if
  $\varepsilon$ is small enough, there exists a point set $P$ whose
  Delaunay complex triangulates $M$. Let $G$ be the $1$-skeleton of this
  complex, and $\gamma$ be a closed curve embedded in $G$.

  By property (i), neighboring points in $G$ are at distance no more
  than $ 2\varepsilon$, therefore we have $|\gamma|_m \leq
  2\varepsilon |\gamma|_G$. There just remains to estimate the value
  of $\varepsilon$, which we do by estimating the number of balls. By
  compactness, the scalar curvature of $M$ is bounded from above by
  some constant $K$. Now, if $\varepsilon$ is small enough, for any $p
  \in P$ we have:

  \[\vol(B(p,\varepsilon/5))\geq
  \frac{\varepsilon^d}{5^d}\left(1-\frac{\varepsilon^2}{6d}K+o(\varepsilon^4)\right)\vol(B^d),\]
  where $B^d$ is the unit Euclidean ball of dimension $d$. This follows
  from standard estimates on the volume of a ball in a Riemannian manifold,
  see for example Gromov~\cite[p. 89]{g-scmc-91}. We recall that
  $\vol(B^d)=\frac{\pi^{d/2}}{\Gamma(d/2+1)}$.

  By property (ii), the balls $B(p,\varepsilon/5)$ are disjoint,
  therefore their number $f_0$ is at most $\frac{ \Gamma(d/2+1)5^d
    V}{\pi^{d/2}\varepsilon^d(1-\varepsilon)}$ if $\varepsilon$ is
  small enough. Finally, putting together our estimates, we obtain
  that
  \[|\gamma|_m\leq (1+\delta)\frac{10}{\sqrt{\pi}} \left(\frac{
      \Gamma(d/2+1)}{f_0} V \right )^{1/d} |\gamma|_G.\qedhere\]
\end{proof}
However, this theorem leads to no immediate corollaries, since unlike the
two-dimensional case, we do not know of any discrete systolic inequalities
involving $f_0$ in dimensions larger than two. This leads to the following
question.

\begin{question}
  Are there manifolds $M$ of dimension $d\geq 3$ for which there exists a
  constant $c_M$ such that, for every triangulation $(M,T)$, there is a
  non-contractible closed curve in the $1$-skeleton of $T$ of length at
  most $c_M f_0(T)^{1/d}$?
\end{question}
  
Notice that a positive answer to this question for essential compact
manifolds without boundary would yield a new proof of Gromov's systolic
inequality.

\medskip

\textit{Remark:} In his thesis~\cite{k-dsi-13}, Kowalick states a theorem
that is closely related to our Theorem~\ref{T:discrtocont-hd}, and thus to
this question. Essentially, his result is ours substituting $f_0$
with $f_d$. Precisely, he shows that if for a manifold $M$, there exists
$c_M'>0$ such that, for every triangulation $(M,T)$, there is a
non-contractible closed curve in the $1$-skeleton of $T$ of length at most
$c_M' f_d(T)^{1/d}$, then there exists a constant $s_M$ such that the
systole of every Riemannian metric on $M$ is bounded above by $s_M
\vol(M)^{1/d}$. This statement can be derived from our
proof without much extra difficulty.  It is enough to show that in the
triangulations constructed in the proof of Theorem~\ref{T:discrtocont-hd},
for large enough $f_0(T)$, the number $f_d(T)$ is bounded above by
$f_0(T)$ up to a multiplicative constant that depends only on the
dimension. This follows again from a packing argument and from the bounds
on volume growth of small balls. For $\epsilon$ small enough with respect
to the strong convexity radius and the minimal sectional curvature of the
manifold, the quotient $\vol(B(x,2\epsilon))/\vol(B(x,\epsilon)$ is bounded
from above by an absolute constant $k_d$ depending only on the
dimension. Since two points in the same facet of our Delaunay complex are
at distance at most $2\epsilon$, we have $f_d(T)\leq \frac{{k_d}^d}{d+1}
f_0(T)$.

\section{Lengths of Genus Zero Decompositions}\label{A:genuszero}

A genus zero decomposition of a surface is a family of disjoint simple
closed curves that cut the surface into a (connected) genus zero surface
with boundary.  Every genus zero decomposition (of a surface with genus at
least two) can be extended to a pants decomposition.  In this section, we
prove the following strengthening of Theorem~\ref{T:gpy}:
\begin{theorem}\label{T:genuszero}
  For any $\varepsilon>0$, the following holds with probability tending to
  one as $n$ tends to~$\infty$: A random (trivalent, unweighted)
  cross-metric surface with $n$~vertices has no genus zero decomposition of
  length at most $n^{7/6-\varepsilon}$.
\end{theorem}
The argument is very similar to the one by Poon and
Thite~\cite[Sect.~2]{pt-pdpp-06}.  As we shall see, this theorem is an
immediate consequence of the following proposition:
\begin{proposition}\label{P:genuszero}
  Let $(S,G^*)$ be a (trivalent, unweighted) cross-metric surface with
  genus zero and $b\ge3$ boundary components.  Then there exists some pants
  decomposition~$\Gamma$ of~$S$ such that each edge of~$G^*$ has $O(\log
  b)$ crossings with each edge of~$\Gamma$.
\end{proposition}
\begin{figure}[htb]
\centering
\inkfrag{a}{(a)}\inkfrag{b}{(b)}\inkfrag{c}{(c)}\inkfrag{d}{(d)}
  \begin{inkfragenv}\def\svgwidth{.9\linewidth}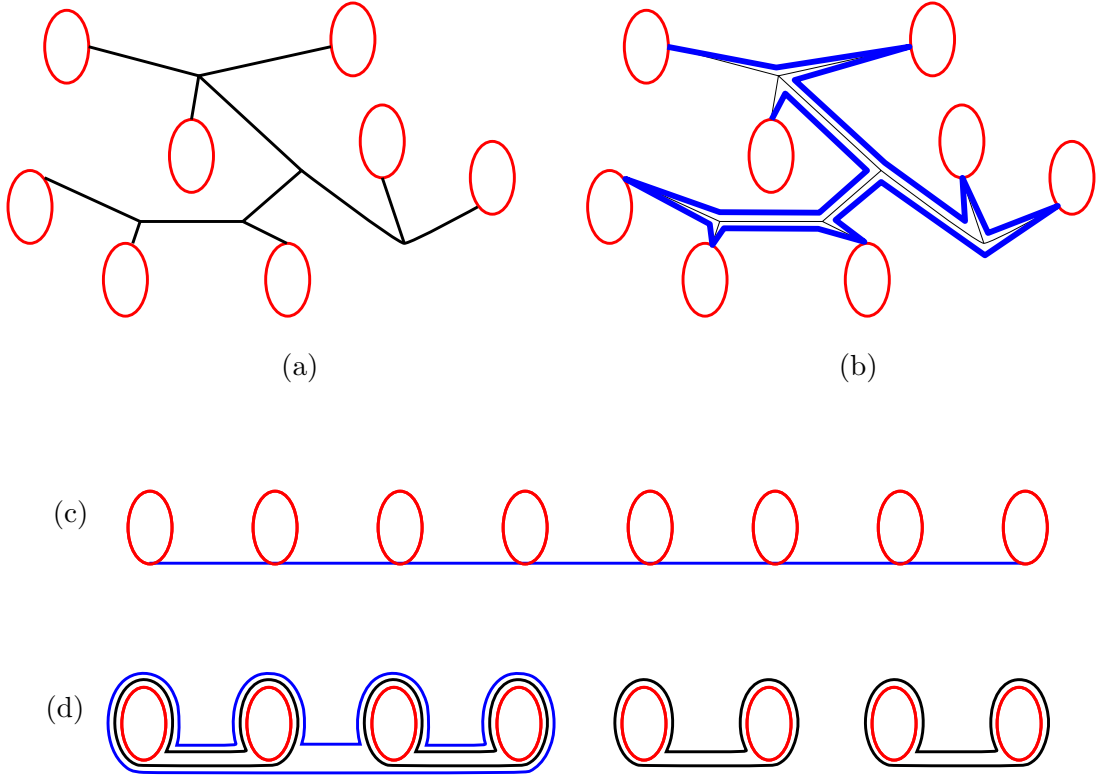\end{inkfragenv}%

\caption{The construction of the pants decomposition in
  Proposition~\ref{P:genuszero}.  (a) The tree~$T$.  (b) The path~$p$.  (c)
  An isomorphic drawing of~$p$.  (d) The pants decomposition.}
\label{F:genuszero}
\end{figure}
\begin{proof}
  Define the \emph{multiplicity} of a set of curves on~$(S,G^*)$ to be the
  maximum number of crossings between an edge of~$G^*$ and the set of curves.

  Let $T$ be a spanning tree of the boundary components of $(S,G^*)$, that
  is, a tree of multiplicity one in $(S,G^*)$ so that each boundary
  component of~$S$ is intersected by exactly one leaf of the tree,
  (Figure~\ref{F:genuszero}(a)). Draw a path~$p$ following the tree~$T$,
  touching it only at the leaves (Figure~\ref{F:genuszero}(b--c)); such a
  path~$p$ has multiplicity two, and touches each boundary component
  exactly once. Let $B_1,B_2,\ldots,B_b$ be the boundary components in
  order along $p$ (oriented arbitrarily).

  Now, we build the pants decomposition (Figure~\ref{F:genuszero}(d)). First we 
  group the boundary components by pairs, $\{B_1,B_2\}$, $\{B_3,B_4\}$,
  and so on. Then we cut~$S$ into a collection of $\lfloor{b/2}\rfloor$ pairs
  of pants and a genus zero surface with $\lceil{b/2}\rceil$ boundary
  components, and we reiterate the process on the latter surface.  After
  $O(\log b)$~iterations, the remaining surface has at most three boundary
  components, so we have built a pants decomposition~$\Gamma$.

  We claim that $\Gamma$ has multiplicity $O(\log b)$.  Indeed, each closed
  curve of~$\Gamma$ is made of (1) pieces that go around a boundary
  component, and (2) pieces that follow a subpath of~$p$.  The pieces of
  type~(1) have overall multiplicity $O(\log b)$, because $O(\log b)$
  pieces go around a given boundary component and each edge of~$G^*$ is
  incident to at most two boundary components.  The pieces of type~(2) have
  overall multiplicity $O(\log b)$, since $O(\log b)$ pieces run along a
  given subpath of~$p$, and because $p$ has multiplicity two in~$(S,G^*)$.
  The result follows.
\end{proof}
\begin{proof}[Proof of Theorem~\ref{T:genuszero}]
  Consider a random cross-metric surface $(S,G^*)$ with $n$~vertices; let
  $g$ be its genus.
  \begin{itemize}
  \item It may be that $(S,G^*)$ has genus zero or one; but this happens
    with probability arbitrarily close to zero, provided $n$ is large
    enough (this follows by combining Lemma~\ref{L:count-connected} with
    Guth et al.~\cite[Lemma~4]{gpy-pdrs-11});
  \item otherwise, if $(S,G^*)$ admits a genus zero decomposition~$\Gamma'$
    of length at most $n^{7/6-\varepsilon}$, we cut $(S,G^*)$
    along~$\Gamma'$, obtaining a cross-metric surface with genus zero with
    $2g\ge3$ boundary components and $O(n^{7/6-\varepsilon})$ edges.
    Proposition~\ref{P:genuszero} implies that this new cross-metric
    surface has a pants decomposition~$\Gamma$ with length
    $O(n^{7/6-\varepsilon}\log g)=O(n^{7/6-\varepsilon}\log n)$.  The union
    of $\Gamma$ and~$\Gamma'$ is a pants decomposition of $(S,G^*)$ of
    length at most $O(n^{7/6-\varepsilon}\log
    n+n^{7/6-\varepsilon})=O(n^{7/6-\varepsilon'})$ for some
    $\varepsilon'<\varepsilon$ if $n$ is large enough.  By
    Theorem~\ref{T:gpy} above, we conclude that this happens with
    arbitrarily small probability as $n\to\infty$.\qedhere
  \end{itemize}
\end{proof}
\end{document}

%% file: fig/counterex.pdf_tex
\begingroup%
  \makeatletter%
  \providecommand\color[2][]{%
    \errmessage{(Inkscape) Color is used for the text in Inkscape, but the package 'color.sty' is not loaded}%
    \renewcommand\color[2][]{}%
  }%
  \providecommand\transparent[1]{%
    \errmessage{(Inkscape) Transparency is used (non-zero) for the text in Inkscape, but the package 'transparent.sty' is not loaded}%
    \renewcommand\transparent[1]{}%
  }%
  \providecommand\rotatebox[2]{#2}%
  \ifx\svgwidth\undefined%
    \setlength{\unitlength}{3160.75bp}%
    \ifx\svgscale\undefined%
      \relax%
    \else%
      \setlength{\unitlength}{\unitlength * \real{\svgscale}}%
    \fi%
  \else%
    \setlength{\unitlength}{\svgwidth}%
  \fi%
  \global\let\svgwidth\undefined%
  \global\let\svgscale\undefined%
  \makeatother%
  \begin{picture}(1,0.36485593)%
    \put(0,0){\includegraphics[width=\unitlength]{counterex.pdf}}%
    \put(0.09615368,0.00411116){\color[rgb]{0,0,0}\makebox(0,0)[lb]{\smash{$\ell$}}}%
    \put(0.41421232,0.00411116){\color[rgb]{0,0,0}\makebox(0,0)[lb]{\smash{$1$}}}%
    \put(0.63818288,0.00411116){\color[rgb]{0,0,0}\makebox(0,0)[lb]{\smash{$\ell$}}}%
    \put(0.95624153,0.00104307){\color[rgb]{0,0,0}\makebox(0,0)[lb]{\smash{$1$}}}%
    \put(0.20865031,0.35694147){\color[rgb]{0,0,0}\makebox(0,0)[lb]{\smash{$\gamma$}}}%
    \put(0.76806533,0.05115519){\color[rgb]{0,0,0}\makebox(0,0)[lb]{\smash{$\alpha$}}}%
    \put(0.74658871,0.08490417){\color[rgb]{0,0,0}\makebox(0,0)[lb]{\smash{$\beta$}}}%
  \end{picture}%
\endgroup%

%% file: fig/pushing.pdf_tex
\begingroup%
  \makeatletter%
  \providecommand\color[2][]{%
    \errmessage{(Inkscape) Color is used for the text in Inkscape, but the package 'color.sty' is not loaded}%
    \renewcommand\color[2][]{}%
  }%
  \providecommand\transparent[1]{%
    \errmessage{(Inkscape) Transparency is used (non-zero) for the text in Inkscape, but the package 'transparent.sty' is not loaded}%
    \renewcommand\transparent[1]{}%
  }%
  \providecommand\rotatebox[2]{#2}%
  \ifx\svgwidth\undefined%
    \setlength{\unitlength}{2153.43488bp}%
    \ifx\svgscale\undefined%
      \relax%
    \else%
      \setlength{\unitlength}{\unitlength * \real{\svgscale}}%
    \fi%
  \else%
    \setlength{\unitlength}{\svgwidth}%
  \fi%
  \global\let\svgwidth\undefined%
  \global\let\svgscale\undefined%
  \makeatother%
  \begin{picture}(1,0.49403016)%
    \put(0,0){\includegraphics[width=\unitlength]{pushing.pdf}}%
    \put(0.22762468,0.00309099){\color[rgb]{0,0,0}\makebox(0,0)[lb]{\smash{(e)}}}%
    \put(0.82397206,0.00309099){\color[rgb]{0,0,0}\makebox(0,0)[lb]{\smash{(g)}}}%
    \put(0.52742122,0.00309099){\color[rgb]{0,0,0}\makebox(0,0)[lb]{\smash{(f)}}}%
    \put(0.24692963,0.17896282){\color[rgb]{0,0,0}\makebox(0,0)[lb]{\smash{$G^*$}}}%
    \put(0.55014383,0.18061978){\color[rgb]{0,0,0}\makebox(0,0)[lb]{\smash{$G^*$}}}%
    \put(0.8445209,0.18282909){\color[rgb]{0,0,0}\makebox(0,0)[lb]{\smash{$G^*$}}}%
    \put(0.09006441,0.27372887){\color[rgb]{0,0,0}\makebox(0,0)[lb]{\smash{(a)}}}%
    \put(0.10936935,0.44960072){\color[rgb]{0,0,0}\makebox(0,0)[lb]{\smash{$G^*$}}}%
    \put(0.35076052,0.27372887){\color[rgb]{0,0,0}\makebox(0,0)[lb]{\smash{(b)}}}%
    \put(0.37006545,0.44960072){\color[rgb]{0,0,0}\makebox(0,0)[lb]{\smash{$G^*$}}}%
    \put(0.60645939,0.47169361){\color[rgb]{0,0,0}\makebox(0,0)[lb]{\smash{$G^*$}}}%
    \put(0.8550044,0.48274005){\color[rgb]{0,0,0}\makebox(0,0)[lb]{\smash{$G^*$}}}%
    \put(0.63907273,0.26931029){\color[rgb]{0,0,0}\makebox(0,0)[lb]{\smash{(c)}}}%
    \put(0.89645489,0.26599636){\color[rgb]{0,0,0}\makebox(0,0)[lb]{\smash{(d)}}}%
  \end{picture}%
\endgroup%

%% file: fig/shifting.pdf_tex
\begingroup%
  \makeatletter%
  \providecommand\color[2][]{%
    \errmessage{(Inkscape) Color is used for the text in Inkscape, but the package 'color.sty' is not loaded}%
    \renewcommand\color[2][]{}%
  }%
  \providecommand\transparent[1]{%
    \errmessage{(Inkscape) Transparency is used (non-zero) for the text in Inkscape, but the package 'transparent.sty' is not loaded}%
    \renewcommand\transparent[1]{}%
  }%
  \providecommand\rotatebox[2]{#2}%
  \ifx\svgwidth\undefined%
    \setlength{\unitlength}{2261.05839844bp}%
    \ifx\svgscale\undefined%
      \relax%
    \else%
      \setlength{\unitlength}{\unitlength * \real{\svgscale}}%
    \fi%
  \else%
    \setlength{\unitlength}{\svgwidth}%
  \fi%
  \global\let\svgwidth\undefined%
  \global\let\svgscale\undefined%
  \makeatother%
  \begin{picture}(1,0.17807233)%
    \put(0,0){\includegraphics[width=\unitlength]{shifting.pdf}}%
    \put(0.15694807,0.00186583){\color[rgb]{0,0,0}\makebox(0,0)[lb]{\smash{(a)}}}%
    \put(0.32940667,0.11949538){\color[rgb]{0,0,0}\makebox(0,0)[lb]{\smash{$G^*$}}}%
    \put(0.01271857,0.16358572){\color[rgb]{0,0,0}\makebox(0,0)[lb]{\smash{$\gamma^i_c$}}}%
    \put(0.85490064,0.05193662){\color[rgb]{0,0,0}\makebox(0,0)[lb]{\smash{$\gamma^i_{c+1}$}}}%
    \put(0.6717422,0.00186583){\color[rgb]{0,0,0}\makebox(0,0)[lb]{\smash{(b)}}}%
    \put(0.8442008,0.11949538){\color[rgb]{0,0,0}\makebox(0,0)[lb]{\smash{$G^*$}}}%
  \end{picture}%
\endgroup%

%% file: fig/JoinSplit.pdf_tex
\begingroup%
  \makeatletter%
  \providecommand\color[2][]{%
    \errmessage{(Inkscape) Color is used for the text in Inkscape, but the package 'color.sty' is not loaded}%
    \renewcommand\color[2][]{}%
  }%
  \providecommand\transparent[1]{%
    \errmessage{(Inkscape) Transparency is used (non-zero) for the text in Inkscape, but the package 'transparent.sty' is not loaded}%
    \renewcommand\transparent[1]{}%
  }%
  \providecommand\rotatebox[2]{#2}%
  \ifx\svgwidth\undefined%
    \setlength{\unitlength}{4014.01068655bp}%
    \ifx\svgscale\undefined%
      \relax%
    \else%
      \setlength{\unitlength}{\unitlength * \real{\svgscale}}%
    \fi%
  \else%
    \setlength{\unitlength}{\svgwidth}%
  \fi%
  \global\let\svgwidth\undefined%
  \global\let\svgscale\undefined%
  \makeatother%
  \begin{picture}(1,0.31473638)%
    \put(0,0){\includegraphics[width=\unitlength]{JoinSplit.pdf}}%
    \put(-0.00066174,0.10675985){\color[rgb]{0,0,0}\makebox(0,0)[lb]{\smash{$\gamma^1$}}}%
    \put(0.17244612,0.0908157){\color[rgb]{0,0,0}\makebox(0,0)[lb]{\smash{$\gamma^1_{r}$}}}%
    \put(0.17301557,0.03842776){\color[rgb]{0,0,0}\makebox(0,0)[lb]{\smash{$\gamma^2_{r}$}}}%
    \put(0.00104653,0.02703909){\color[rgb]{0,0,0}\makebox(0,0)[lb]{\smash{$\gamma^2$}}}%
    \put(0.07678126,0.10391268){\color[rgb]{0,0,0}\makebox(0,0)[lb]{\smash{$\gamma^1_{s}$}}}%
    \put(0.07905897,0.03102512){\color[rgb]{0,0,0}\makebox(0,0)[lb]{\smash{$\gamma^2_{s}$}}}%
    \put(0.13258579,0.0959406){\color[rgb]{0,0,0}\makebox(0,0)[lb]{\smash{$\eta$}}}%
    \put(0.68459123,0.10817507){\color[rgb]{0,0,0}\makebox(0,0)[lb]{\smash{$\gamma^1$}}}%
    \put(0.68345235,0.03414862){\color[rgb]{0,0,0}\makebox(0,0)[lb]{\smash{$\gamma^2$}}}%
    \put(0.8542826,0.0705924){\color[rgb]{0,0,0}\makebox(0,0)[lb]{\smash{$\delta$}}}%
    \put(0.48364786,0.25040415){\color[rgb]{0,0,0}\makebox(0,0)[lb]{\smash{(a)}}}%
    \put(0.70949737,0.28984768){\color[rgb]{0,0,0}\makebox(0,0)[lb]{\smash{$\gamma^1$}}}%
    \put(0.79605138,0.28813938){\color[rgb]{0,0,0}\makebox(0,0)[lb]{\smash{$\delta_1$}}}%
    \put(0.79662077,0.17140539){\color[rgb]{0,0,0}\makebox(0,0)[lb]{\smash{$\delta_2$}}}%
    \put(0.00813138,0.29102702){\color[rgb]{0,0,0}\makebox(0,0)[lb]{\smash{$\gamma^1$}}}%
    \put(0.16130916,0.30867948){\color[rgb]{0,0,0}\makebox(0,0)[lb]{\smash{$\gamma^1_{r}$}}}%
    \put(0.08956043,0.2938742){\color[rgb]{0,0,0}\makebox(0,0)[lb]{\smash{$\gamma^1_{s}$}}}%
    \put(0.1356846,0.25059721){\color[rgb]{0,0,0}\makebox(0,0)[lb]{\smash{$\eta$}}}%
    \put(0.19376688,0.26198588){\color[rgb]{0,0,0}\makebox(0,0)[lb]{\smash{$\alpha$}}}%
    \put(0.19718348,0.19934814){\color[rgb]{0,0,0}\makebox(0,0)[lb]{\smash{$\beta$}}}%
    \put(0.48356222,0.09459021){\color[rgb]{0,0,0}\makebox(0,0)[lb]{\smash{(b)}}}%
  \end{picture}%
\endgroup%

%% file: fig/count-schemas-detail-2a.pdf_tex
\begingroup%
  \makeatletter%
  \providecommand\color[2][]{%
    \errmessage{(Inkscape) Color is used for the text in Inkscape, but the package 'color.sty' is not loaded}%
    \renewcommand\color[2][]{}%
  }%
  \providecommand\transparent[1]{%
    \errmessage{(Inkscape) Transparency is used (non-zero) for the text in Inkscape, but the package 'transparent.sty' is not loaded}%
    \renewcommand\transparent[1]{}%
  }%
  \providecommand\rotatebox[2]{#2}%
  \ifx\svgwidth\undefined%
    \setlength{\unitlength}{82bp}%
    \ifx\svgscale\undefined%
      \relax%
    \else%
      \setlength{\unitlength}{\unitlength * \real{\svgscale}}%
    \fi%
  \else%
    \setlength{\unitlength}{\svgwidth}%
  \fi%
  \global\let\svgwidth\undefined%
  \global\let\svgscale\undefined%
  \makeatother%
  \begin{picture}(1,1.28591074)%
    \put(0,0){\includegraphics[width=\unitlength]{count-schemas-detail-2a.pdf}}%
    \put(0.45225235,1.20056394){\color[rgb]{0,0,0}\makebox(0,0)[lb]{\smash{p1}}}%
    \put(0.44535374,0.02435213){\color[rgb]{0,0,0}\makebox(0,0)[lb]{\smash{p2}}}%
    \put(0.2039026,0.62108132){\color[rgb]{0,0,0}\makebox(0,0)[lb]{\smash{c1}}}%
    \put(0.70060202,0.61418271){\color[rgb]{0,0,0}\makebox(0,0)[lb]{\smash{c2}}}%
  \end{picture}%
\endgroup%

%% file: fig/count-schemas-detail-2b.pdf_tex
\begingroup%
  \makeatletter%
  \providecommand\color[2][]{%
    \errmessage{(Inkscape) Color is used for the text in Inkscape, but the package 'color.sty' is not loaded}%
    \renewcommand\color[2][]{}%
  }%
  \providecommand\transparent[1]{%
    \errmessage{(Inkscape) Transparency is used (non-zero) for the text in Inkscape, but the package 'transparent.sty' is not loaded}%
    \renewcommand\transparent[1]{}%
  }%
  \providecommand\rotatebox[2]{#2}%
  \ifx\svgwidth\undefined%
    \setlength{\unitlength}{82bp}%
    \ifx\svgscale\undefined%
      \relax%
    \else%
      \setlength{\unitlength}{\unitlength * \real{\svgscale}}%
    \fi%
  \else%
    \setlength{\unitlength}{\svgwidth}%
  \fi%
  \global\let\svgwidth\undefined%
  \global\let\svgscale\undefined%
  \makeatother%
  \begin{picture}(1,1.01475919)%
    \put(0,0){\includegraphics[width=\unitlength]{count-schemas-detail-2b.pdf}}%
  \end{picture}%
\endgroup%

%% file: fig/count-schemas-detail-2c.pdf_tex
\begingroup%
  \makeatletter%
  \providecommand\color[2][]{%
    \errmessage{(Inkscape) Color is used for the text in Inkscape, but the package 'color.sty' is not loaded}%
    \renewcommand\color[2][]{}%
  }%
  \providecommand\transparent[1]{%
    \errmessage{(Inkscape) Transparency is used (non-zero) for the text in Inkscape, but the package 'transparent.sty' is not loaded}%
    \renewcommand\transparent[1]{}%
  }%
  \providecommand\rotatebox[2]{#2}%
  \ifx\svgwidth\undefined%
    \setlength{\unitlength}{78.03879199bp}%
    \ifx\svgscale\undefined%
      \relax%
    \else%
      \setlength{\unitlength}{\unitlength * \real{\svgscale}}%
    \fi%
  \else%
    \setlength{\unitlength}{\svgwidth}%
  \fi%
  \global\let\svgwidth\undefined%
  \global\let\svgscale\undefined%
  \makeatother%
  \begin{picture}(1,1.05271015)%
    \put(0,0){\includegraphics[width=\unitlength]{count-schemas-detail-2c.pdf}}%
  \end{picture}%
\endgroup%

%% file: fig/polyhedral2.pdf_tex
\begingroup%
  \makeatletter%
  \providecommand\color[2][]{%
    \errmessage{(Inkscape) Color is used for the text in Inkscape, but the package 'color.sty' is not loaded}%
    \renewcommand\color[2][]{}%
  }%
  \providecommand\transparent[1]{%
    \errmessage{(Inkscape) Transparency is used (non-zero) for the text in Inkscape, but the package 'transparent.sty' is not loaded}%
    \renewcommand\transparent[1]{}%
  }%
  \providecommand\rotatebox[2]{#2}%
  \ifx\svgwidth\undefined%
    \setlength{\unitlength}{328.09715024bp}%
    \ifx\svgscale\undefined%
      \relax%
    \else%
      \setlength{\unitlength}{\unitlength * \real{\svgscale}}%
    \fi%
  \else%
    \setlength{\unitlength}{\svgwidth}%
  \fi%
  \global\let\svgwidth\undefined%
  \global\let\svgscale\undefined%
  \makeatother%
  \begin{picture}(1,1.08703649)%
    \put(0,0){\includegraphics[width=\unitlength]{polyhedral2.pdf}}%
    \put(0.23417598,0.74888133){\color[rgb]{0,0,0}\makebox(0,0)[lb]{\smash{a}}}%
    \put(0.76296249,0.74888133){\color[rgb]{0,0,0}\makebox(0,0)[lb]{\smash{b}}}%
    \put(0.23509984,0.37480997){\color[rgb]{0,0,0}\makebox(0,0)[lb]{\smash{c}}}%
    \put(0.75992523,0.37624486){\color[rgb]{0,0,0}\makebox(0,0)[lb]{\smash{d}}}%
    \put(0.23210204,0.00029125){\color[rgb]{0,0,0}\makebox(0,0)[lb]{\smash{e}}}%
    \put(0.75892257,0.00015314){\color[rgb]{0,0,0}\makebox(0,0)[lb]{\smash{f}}}%
  \end{picture}%
\endgroup%

%% file: fig/genuszero.pdf_tex
\begingroup%
  \makeatletter%
  \providecommand\color[2][]{%
    \errmessage{(Inkscape) Color is used for the text in Inkscape, but the package 'color.sty' is not loaded}%
    \renewcommand\color[2][]{}%
  }%
  \providecommand\transparent[1]{%
    \errmessage{(Inkscape) Transparency is used (non-zero) for the text in Inkscape, but the package 'transparent.sty' is not loaded}%
    \renewcommand\transparent[1]{}%
  }%
  \providecommand\rotatebox[2]{#2}%
  \ifx\svgwidth\undefined%
    \setlength{\unitlength}{593.5356bp}%
    \ifx\svgscale\undefined%
      \relax%
    \else%
      \setlength{\unitlength}{\unitlength * \real{\svgscale}}%
    \fi%
  \else%
    \setlength{\unitlength}{\svgwidth}%
  \fi%
  \global\let\svgwidth\undefined%
  \global\let\svgscale\undefined%
  \makeatother%
  \begin{picture}(1,0.71512812)%
    \put(0,0){\includegraphics[width=\unitlength]{genuszero.pdf}}%
    \put(0.04255596,0.22689358){\color[rgb]{0,0,0}\makebox(0,0)[lb]{\smash{c}}}%
    \put(0.0359093,0.04610374){\color[rgb]{0,0,0}\makebox(0,0)[lb]{\smash{d}}}%
    \put(0.25220631,0.36331573){\color[rgb]{0,0,0}\makebox(0,0)[lb]{\smash{a}}}%
    \put(0.76433861,0.36331573){\color[rgb]{0,0,0}\makebox(0,0)[lb]{\smash{b}}}%
  \end{picture}%
\endgroup%